\documentclass[a4paper]{article}
\usepackage{latexsym,bm}
\usepackage{amssymb,amsmath,amsthm}
\usepackage[english,german]{babel}
\usepackage[colorlinks=true]{hyperref}
\usepackage{color}
\usepackage{times}
\allowdisplaybreaks[4]

\newtheorem{theorem}{Theorem}[section]
\newtheorem{lemma}[theorem]{Lemma}

\newtheorem{remark}[theorem]{Remark}
\newtheorem{proposition}[theorem]{Proposition}

\newtheorem{corollary}[theorem]{Corollary}
\numberwithin{equation}{section}
\hypersetup{linkcolor=blue,urlcolor=red,citecolor=red}
\title{Observability from measurable sets for strongly coupled parabolic systems via single-component observation\thanks{This work was partially supported by the National Key R\&D Program of China under grant 2024YFA1012802, the New Cornerstone Science Foundation, and by the National Natural Science Foundation of China  under grants 12471424, 12371450, 12571483.}}
\author{Xiaoyu Fu\thanks{School of Mathematics, Sichuan University, Chengdu, 610064, China (xiaoyufu@scu.edu.cn).}\and Gengsheng Wang\thanks{Center for Applied Mathematics, Tianjin University, Tianjin, 300072, China (wanggs@yeah.net).}\and Huaiqiang Yu\thanks{School of Mathematics, Tianjin University, Tianjin, 300354, China (huaiqiangyu@yeah.net).}
\and Xiaomin Zhu\thanks{Civil Aviation Flight University of China, Deyang, 618307, China (xiaominzhu123@yeah.net).}}
\date{}
\begin{document}
\selectlanguage{english}
\maketitle
\begin{abstract}

We establish an observability inequality from space-time measurable sets for a class of strongly coupled parabolic systems consisting of two equations, where the observation acts on a single-component. The model is motivated by parabolic equations with complex coefficients and serves as a prototypical example of strongly coupled systems.
The main difficulty lies in the fact that, unlike in the scalar and weakly coupled cases, pointwise-in-time interpolation observability estimates fail, as the observed component may exhibit high-frequency oscillatory cancellations induced by the coupling. To overcome this difficulty, we develop a new integral-type interpolation observability inequality based on a Remez-type inequality.
With the aid of this integral-type interpolation observability inequality and the strategy developed in \cite{PW, Luis-Wang-Zhang} for deriving observability from measurable sets, we obtain the desired observability inequality.
\end{abstract}



\section{Introduction}\label{yu-sect-intro}

Let  $\Omega$  be a bounded Lipschitz, locally star-shaped domain
 in $\mathbb{R}^d$ with boundary $\partial\Omega$. Let  $T>0$, and let $a, b\in \mathbb{R}$  satisfy
$a>0$ and $b\neq 0$. We consider the following strongly coupled parabolic system:
 \begin{equation}\label{yu-8-28-1}
\begin{cases}
    \partial_ty_1=a\Delta y_1-b\Delta y_2&\mbox{in}\;\;\Omega\times(0,T),\\
    \partial_ty_2=b\Delta y_1+a\Delta y_2 &\mbox{in}\;\;\Omega\times(0,T),\\
     y_1=y_2=0&\mbox{on}\;\;\partial\Omega\times(0,T).
\end{cases}
\end{equation}
    System \eqref{yu-8-28-1} admits a unique solution in $C([0,T];L^2(\Omega;\mathbb{R}^2))$ for each initial state $y_0:=(y_1(0),y_2(0))^\top\in L^2(\Omega;\mathbb{R}^2)$.
     We denote this solution by $y(\cdot,\cdot;y_0):=(y_1(\cdot,\cdot;y_0),y_2(\cdot,\cdot;y_0))^\top$, viewed  as a function from $\Omega\times[0, T]$ to $\mathbb{R}^2$, or  by  $y(\cdot;y_0):=(y_1(\cdot;y_0),y_2(\cdot;y_0))^\top$, regarded as a function from $[0, T]$ to $L^2(\Omega;\mathbb{R}^2)$.

 The aim of this paper is to establish the following observability inequality:
\begin{equation}\label{yu-8-28-2}
\|y(T;y_0)\|_{L^2(\Omega;\mathbb{R}^2)}\leq N\int_{\Omega\times(0,T)}\chi_{\mathcal{D}}(x,t)|y_1(x,t;y_0)|dxdt
\quad\text{for all } y_0\in L^2(\Omega;\mathbb{R}^2),
\end{equation}
 where $\mathcal{D}\subset\Omega\times(0,T)$ is a set of positive measure, and  $N:=N(a,b,T,\mathcal{D})>0$ is a constant.
 Inequality \eqref{yu-8-28-2} allows one to quantitatively reconstruct the full system state from a single-component observation over a space-time measurable subset.

One  motivation for studying \eqref{yu-8-28-2} arises from parabolic equations with complex coefficients.
 A typical example is
  the linearized Ginzburg-Landau equation around the steady state $0$:
\begin{equation}\label{yu-8-20-1}
\begin{cases}
     z_t=(a+\mathrm{i}b)\Delta_{\mathbb{C}} z&\mbox{in}\;\;\Omega\times(0,T),\\
    z=0&\mbox{on}\;\;\partial\Omega\times(0,T),
\end{cases}
\end{equation}
   where $\Delta_{\mathbb{C}}$ denotes the complex Laplacian
    on $L^2(\Omega;\mathbb{C})$.
    For this equation, observability inequality \eqref{yu-8-28-2} takes the following form:
\begin{equation*}\label{yu-8-20-3}
    \|z(T;z_0)\|_{L^2(\Omega;\mathbb{C})}\leq N\int_{\Omega\times(0,T)}\chi_{\mathcal{D}}|\textsf{Re}\,z(x,t;z_0)|dxdt
    \;\;\mbox{for all}\;\;z_0\in L^2(\Omega;\mathbb{C}),
\end{equation*}
where $z(\cdot;z_0)$ denotes the solution corresponding to the initial state $z_0$.
 Another important  motivation for studying (\ref{yu-8-28-2}) comes from an open problem raised in \cite{EMZ}:  \emph{can strongly coupled parabolic systems be observed from space-time
measurable sets under partial (single-component) observation?}

Observability inequalities on space-time measurable sets have been studied for scalar parabolic equations, largely motivated by time-optimal control problems (see \cite{Luis-Wang-Zhang, PW, Phung-Wang-Zhang} and references therein).

For coupled parabolic systems in which the number of observed components is less than the number of state components, existing observability results fall into two categories. The first concerns  strongly coupled systems, but the observation domains  are cylindrical sets of the form
$\omega\times(0,T)$, where  $\omega$ is an open subset of $\Omega$
 (see \cite{FGT1, Burgos-Garcia, G, Lissy-Zuazua} and references therein).  The second category deals with weakly coupled systems, i.e., systems whose principal-part operators are uncoupled (see \cite{Ammar-jmpa, AmmarSurvey, EMZ, FGT, Qin-Wang-Yu, Wang-Yan-Yu}). To the best of our knowledge,
   observability from general measurable sets for  strongly coupled parabolic systems under single-component
    observation has not been established by any existing method. The present work provides the first positive result in this setting.

The observability of system \eqref{yu-8-28-1} under single-component observation is significantly more challenging than its scalar counterpart, due to two main difficulties:
\begin{itemize}
  \item For the observability of system \eqref{yu-8-28-1} under single-component observation, the observed component may exhibit cancellations of high-frequency oscillations, which may cause the pointwise-in-time observation signal to vanish. Consequently, the single-time-point interpolation observability inequality used in \cite{Luis-Wang-Zhang, EMZ, PW, Phung-Wang-Zhang}, which is crucial for scalar parabolic equations with observations on measurable sets, does not hold, as shown in Proposition \ref{yu-proposition-9-2-1}. Even the multi-time-point interpolation observability inequality, established in \cite{Wang-Yan-Yu} for weakly coupled parabolic systems, may fail, as demonstrated in Proposition \ref{yu-proposition-4-8-1}. This difficulty stems from the intrinsically strong coupling structure of system \eqref{yu-8-28-1}.
To overcome this difficulty, we establish an integral-type interpolation observability inequality with the aid of a Remez-type inequality, rather than relying on pointwise-in-time interpolation observability inequalities.

  \item  When  $\mathcal{D}=\omega\times(0,T)$ with $\omega\subset\Omega$
 an open subset,
inequality \eqref{yu-8-28-2} was established in \cite{Lissy-Zuazua}.
Their approach combines the Lebeau-Robbiano inequality with an observability estimate for linear ODE systems, which requires the observability constant  $C(T)$
 to have a specific dependence on $T$.
However, this approach does not apply to our setting, as $\mathcal{D}\subset \Omega\times(0,T)$
 is only a measurable set of positive measure.
Instead, we derive inequality \eqref{yu-8-28-2} by combining the strategy for observability from measurable sets developed in \cite{PW, Luis-Wang-Zhang} with the integral-type interpolation observability inequality mentioned above.

 \end{itemize}

Our main result is as follows.

\begin{theorem}\label{yu-theorem-8-18-2}
Let $\mathcal{D}\subset\Omega\times(0,T)$ be a set of positive measure.
 There exists a constant
 $N=N(a, b, T, \mathcal{D})>0$ such that the observability inequality \eqref{yu-8-28-2} holds.
\end{theorem}

\begin{remark}\label{yu-remark-8-28-1}
   Several comments on Theorem \ref{yu-theorem-8-18-2} are in order.
\begin{enumerate}

  \item [$(i)$] In Theorem \ref{yu-theorem-8-18-2}, we only consider the case where the observer acts on the first component of system \eqref{yu-8-28-1}. The same conclusion holds when the observer is imposed on the second component, as shown in Proposition \ref{yu-proposition-9-8-1} of Section \ref{yu-sec-4}.
  \item [$(ii)$] Two extreme cases are noteworthy:
\begin{enumerate}
  \item [$(1)$] If $a=0$ and $b\neq 0$,
  system \eqref{yu-8-28-1} reduces to the  Schr\"{o}dinger equation. Our observability problem then reduces to recovering the full state of the system by observing only the real part of the solution on a measurable space-time set.
Unfortunately, this problem remains open even when  $\mathcal{D}=\omega\times(0,T)$ with $\omega\subset\Omega$
 an open subset. For related results, we refer the reader to \cite{Araruna-Cerpa-Mercado-Santos, Wang-Wang-Zhang} and references therein.

  \item [$(2)$] If $a\neq 0$ and $b=0$,  system \eqref{yu-8-28-1} becomes decoupled.
Consequently, the state of the system cannot be recovered by observing only one component of the solution.

\end{enumerate}
      \item [$(iii)$] As an application of Theorem \ref{yu-theorem-8-18-2}, we establish the bang-bang property for the time-optimal control problem studied in Section \ref{yu-sec-5}.
While the idea of the proof originates from \cite[Theorem 1.2]{Wang2009SICON}, it involves certain technical difficulties that were not emphasized in \cite{Luis-Wang-Zhang}, where the proof was omitted.

\end{enumerate}
\end{remark}

The rest of the paper is organized as follows.
Section~\ref{sec-yu-2} collects the notation  and several preliminary results.
Section~\ref{yu-sec-3} establishes refined integral-type interpolation inequalities for strongly coupled parabolic systems.
Section~\ref{yu-sec-4} is devoted to the proof of Theorem \ref{yu-theorem-8-18-2}.
Section~\ref{yu-sec-5} applies Theorem \ref{yu-theorem-8-18-2} to derive the bang-bang property of time-optimal control problems.

 \section{Preliminaries}\label{sec-yu-2}

We collect here some notation that will be used repeatedly in the sequel. Let $\mathbb{R}^+:=(0,+\infty)$ and $\mathbb{N}^+:=\{1,2,\ldots\}$. For $d\in\mathbb{N}^+$, $r\in\mathbb{R}^+$, and $x_0\in\mathbb{R}^d$, let $B_r^d(x_0)$ denote the closed ball in $\mathbb{R}^d$ centered at $x_0$ with radius $r$. In particular, we write $B_r^d$ when $x_0=0$. For any Lebesgue measurable set $E\subset\mathbb{R}^d$, $|E|$ denotes its Lebesgue measure.
    Given a matrix $M\in\mathbb{R}^{d_1\times d_2}$ with $d_1,d_2\in\mathbb{N}^+$,
    we denote by $M^\top$ its transpose. For a Banach space $X$,
    $\|\cdot\|_X$ denotes its norm. If $X$ is a Hilbert space,  $\langle\cdot,\cdot\rangle_X$ denotes
    its inner product.
    For two Banach spaces $X_1$ and $X_2$, let $\mathcal{L}(X_1;X_2)$ denote the space of all bounded linear operators from $X_1$ to $X_2$, and write $\mathcal{L}(X_1):=\mathcal{L}(X_1;X_1)$.
        For $a_1, a_2\in\mathbb{R}$, we define $a_1\wedge a_2:=\min\{a_1, a_2\}$ and $[a_1]_{\mathbb{Z}}:=\max\{k\in\mathbb{Z}:k\leq a_1\}$.

Let $\{\lambda_i\}_{i\in\mathbb{N}^+}$, with $0<\lambda_1<\lambda_2\leq\cdots$ and $\lambda_k\to+\infty$ as $k\to+\infty$, be the eigenvalues of $-\Delta$ under the homogeneous
 Dirichlet boundary condition.  Let $\{e_i\}_{i\in\mathbb{N}^+}$ be the corresponding normalized eigenfunctions, which form an orthonormal basis of $L^2(\Omega;\mathbb{R})$. We further define
\begin{equation}\label{yu-8-19-00}
k_\lambda:=\max\{i\in\mathbb{N}^+:\lambda_i\leq \lambda\}\;\;\mbox{for any}\;\;
 \lambda > \lambda_1.
\end{equation}
    By the classical Weyl law  (see, e.g.,  \cite{YY}), there exists a constant $M=M(\Omega)>0$ such that
\begin{equation}\label{yu-nn-1-1}
    k_\lambda=M\lambda^{\frac{d}{2}}+o(\lambda^{\frac{d}{2}})\;\;\mbox{as}\;\;\lambda\to+\infty.
\end{equation}

We now recall the  $L^1$-type  Lebeau-Robbiano inequality   for measurable sets (see
   \cite[Theorems 3 and 5]{Luis-Wang-Zhang}).
\begin{lemma}\label{a1}
    Let $x_0\in \Omega$ and $R\in(0,1]$ be such that $B_{4R}^d(x_0)\subset \Omega$. For each measurable set $\omega\subset B_R^d(x_0)$ with $|\omega|>0$, there exists a constant $C(R,|\omega|/|B_R^d|)>0$ such that,
    for all $\lambda>\lambda_1$ and all $\{a_i\}_{i=1}^{k_\lambda}\subset\mathbb{R}$,
\begin{equation}\label{yu-8-19-20}
\sum_{i=1}^{k_\lambda}|a_i|^2 \leq C(R,|\omega|/|B_R^d|) e^{C(R,|\omega|/|B_R^d|) \sqrt{\lambda}} \Big(\Big\|\chi_\omega\sum_{i=1}^{k_\lambda} a_i e_i\Big\|_{L^1(\Omega;\mathbb{R})}\Big)^2.
\end{equation}
\end{lemma}
As a corollary of Lemma \ref{a1}, we obtain the following result, which will be used in the sequel.
\begin{corollary}\label{yu-corollary-8-29-3}
    Suppose that $B^d_{4R}(x_0)\subset \Omega$ for some $R\in\mathbb{R}^+$ and $x_0\in\Omega$. Let
    $D\in(0, |B^d_{R}|)$.
    Then there exists a constant $C(R,D)>0$ such that for every measurable set $\omega\subset B^d_R(x_0)$
    satisfying $|\omega|\geq D$,
       inequality \eqref{yu-8-19-20} holds
            with $C(R,|\omega|/|B_R^d|)$ replaced by $C(R,D)$.
\end{corollary}
\begin{proof}
  Fix an arbitrary  subset $\omega\subset B^d_R(x_0)$ with $|\omega|\geq D$.
    Since $|\omega|\geq D$ and $D\in(0, |B^d_{R}|)$,
    the absolute continuity of the Lebesgue integral ensures the existence of a radius
             $r^*\in (0,R]$ such that
\begin{equation*}
    |B_{r^*}^d(x_0)\cap\omega|=\int_{B^d_{r^*}(x_0)}\chi_{\omega}(x)dx=D.
\end{equation*}
   Let  $\omega_{r^*}:=B_{r^*}^d(x_0)\cap\omega\subset B^d_{R}(x_0)$.
   From the equality above, it follows that $|\omega_{r^*}|/|B^d_R|=D/|B_R^d|$.
   By virtue of this fact, we  apply Lemma \ref{a1} with $\omega$ replaced by $\omega_{r^*}$
    to obtain  a  constant
    $C(R,|\omega_{r^*}|/|B_R^d|)=C(R,D/|B_R^d|)>0$ such that  for all $\lambda> \lambda_1$  and all $\{a_i\}_{i=1}^{k_\lambda}\subset\mathbb{R}$,
    \begin{equation*}
\sum_{i=1}^{k_\lambda}|a_i|^2 \leq C(R,D/|B_R^d|)e^{C(R,D/|B_R^d|)\sqrt{\lambda}} \Big(\Big\|\chi_{\omega_{r^*}}\sum_{i=1}^{k_\lambda} a_i e_i\Big\|_{L^1(\Omega;\mathbb{R})}\Big)^2.
\end{equation*}
   Since $\omega_{r^*}\subset \omega$, the above inequality yields   \eqref{yu-8-19-20}
            with $C(R,|\omega|/|B_R^d|)$ replaced by $C(R,D/|B_R^d|)$.

            Redefine $C(R,D):=C(R,D/|B_R^d|)$. This is well defined since,  for fixed dimension $d$,
            the quantity $D/|B_R^d|$ depends only on $R$ and $D$. Consequently, \eqref{yu-8-19-20}
            holds
            with $C(R,|\omega|/|B_R^d|)$ replaced by $C(R,D)$. This completes the proof.
\end{proof}
The following lemma deals with the partition of measurable sets and will be used repeatedly in the sequel.
\begin{lemma}\label{yu-lemma-9-3-1}
    Let $\mathcal{O}$ be an open set of  $\mathbb{R}^l$ with $l\in\mathbb{N}^+$, and let $Q\subset \mathcal{O}$ be a set of positive Lebesgue measure.
    Then there exist $q\in Q$ and $R>0$ such that, for every $r\in(0,R]$, $B_{4r}^{l}(q)\subset \mathcal{O}$ and $|B_r^l(q)\cap Q|\geq \frac{1}{2}|B_r^l|$.
\end{lemma}
\begin{proof}
    Let
     $q\in Q$ be a Lebesgue point of $\chi_{Q}$. Then $\lim_{r\to 0}\frac{|B_r^l(q)\cap Q|}{|B_r^l|}=1$,
     so
    there exists $R_1>0$ such that
\begin{equation}\label{yu-b-9-3-1}
    |B^l_{r}(q)\cap Q|\geq
    \frac{1}{2}|B^l_{r}|\;\;\mbox{for all}\;\;r\in(0,R_1].
\end{equation}
   Moreover, since $\mathcal{O}$ is open and $q\in Q$, there  exists $R_2>0$ such that
    $B_{4r}^l(q)\subset \mathcal{O}$ for each $r\in(0,R_2]$.
    Combining this with \eqref{yu-b-9-3-1} implies that for
     all $r\in(0,R_1\wedge R_2]$,
    $B_{4r}^{l}(q)\subset \mathcal{O}$ and $|B_r^l(q)\cap Q|\geq \frac{1}{2}|B_r^l|$. Setting $R:=R_1\wedge R_2$ completes the proof.
\end{proof}

 The following is a standard result. We include its proof for the sake of completeness.
\begin{lemma}\label{yu-lemma-9-4-3}
    Let $H$ be a Hilbert space and let $F_i(\cdot)\in C(H; [0, +\infty))$ for $i=1,2,3$. Assume that
     $F_1(h)\leq F_3(h)$ for all $h\in H$. Then for each $\theta\in(0,1)$, the following statements are equivalent:
     \begin{enumerate}
  \item [$(i)$] There exists $\Pi_1\geq 1$ such that $F_1(h)\leq \Pi_1(\varepsilon^{-\gamma}F_2(h)+\varepsilon F_3(h))$
   for all $h\in H$ and all $\varepsilon\in(0,1)$, where $\gamma:=\theta(1-\theta)^{-1}$.
  \item [$(ii)$] There exists $\Pi_2\geq 1$ such that $F_1(h)\leq \Pi_2 (F_2(h))^{1-\theta}(F_3(h))^{\theta}$ for all $h\in H$.
\end{enumerate}
  Moreover, if $(i)$ holds with constant $\Pi_1$, then $(ii)$ holds with $\Pi_2=2\Pi_1$.
\end{lemma}
\begin{proof}
    We arbitrarily fix $\theta\in (0,1)$.
    The implication $(ii) \Rightarrow (i)$  follows immediately from   Young's inequality.

    We now prove  $(i)\Rightarrow(ii)$. Assume that $(i)$ holds with some constant $\Pi_1\ge 1$.
It suffices to prove that
\begin{equation}\label{yu-9-11-30}
    F_1(h)\leq 2\Pi_1 (F_2(h))^{1-\theta}(F_3(h))^{\theta}\;\;\mbox{for all}\;\;h\in H.
\end{equation}
    To this end, we  fix an arbitrary $h\in H$.
    In the case $F_3(h)\leq F_2(h)$,
    inequality \eqref{yu-9-11-30}
         follows directly from the assumption $F_1(h)\leq F_3(h)$. We now consider the case $F_3(h)>F_2(h)$. If $F_2(h)=0$, then from $(i)$, $F_1(h)\le \varepsilon\Pi_1 F_3(h)$ for all $\varepsilon\in(0,1)$. Letting $\varepsilon\to 0^+$, we obtain $F_1(h)=0$, and hence (\ref{yu-9-11-30}) follows. If $F_2(h)>0$, define  $\varepsilon_*:=\big(F_2(h)/F_3(h)\big)^{\frac{1}{\gamma+1}}$, where  $\gamma:=\theta(1-\theta)^{-1}$.
   Then $\varepsilon_*\in(0,1)$ and $(\varepsilon_*)^{-\gamma}F_2(h)=\varepsilon_* F_3(h)$.
   Combining these with  $(i)$ yields \eqref{yu-9-11-30}. Moreover, from \eqref{yu-9-11-30}, one can see that if $(i)$ is true  with the constant $\Pi_1$, then the constant $\Pi_2$ in $(ii)$ can be chosen to be $2\Pi_1$. This completes the proof.
\end{proof}

The following lemma  will be used in the proof of
Theorem~\ref{yu-theorem-8-18-2}.

\begin{lemma}\label{yu-lemma-8-19-1}
    Suppose that $T>0$ and $B^d_R(x_0)\subset\Omega$ for some $R>0$ and $x_0\in\Omega$. Let $\mathcal{D}\subset B^d_R(x_0)\times(0,T)$ be a measurable set with $|\mathcal{D}|>0$. Set
\begin{equation}\label{yu-8-19-23}
    \mathcal{D}_t:=\{x\in\Omega:(x,t)\in \mathcal{D}\}\;\;\mbox{for}\;\;t\in(0,T)\;\;\mbox{and}\;\;
    E:=\Big\{t\in(0,T):|\mathcal{D}_t|\geq \frac{|\mathcal{D}|}{2T}\Big\}.
\end{equation}
    Then, the following  statements hold:
\begin{enumerate}
  \item [$(i)$] $\mathcal{D}_t$ is a measurable subset of $\Omega$ for a.e. $t\in(0,T)$;
  \item[$(ii)$]  for any $f\in L^1(\Omega\times(0,T))$,
      the function $t\mapsto \int_{\Omega}\chi_{\mathcal{D}_t}(x)|f(x,t)|dx$ is integrable over $(0,T)$;
  \item [$(iii)$] $E$ is a measurable subset of $(0,T)$  and $|E|\geq \frac{|\mathcal{D}|}{2|B^d_R|}$;
 \item [$(iv)$] $\chi_{E}\chi_{\mathcal{D}_t}\leq \chi_{\mathcal{D}}$  in $\Omega\times (0,T)$ in the sense
 that $\int_0^T\int_\Omega \chi_{E}\chi_{\mathcal{D}_t}|f|dxdt\leq \int_0^T\int_\Omega \chi_{\mathcal{D}}|f|dxdt$
  for all $f\in L^1(\Omega\times(0,T))$ (here, we note that, by $(i)$ and $(ii)$, the integral on the left-hand side of this inequality is well defined).
\end{enumerate}
\end{lemma}
\begin{proof}
    Statements $(i)$, $(iii)$, and $(iv)$ are proved in \cite[Lemma 1]{Luis-Wang-Zhang}, so it remains to verify $(ii)$, which follows from Fubini’s theorem since $\mathcal{D}_t$ is a slice of $\mathcal{D}$ for each $t \in (0,T)$ and hence $\chi_{\mathcal{D}_t}(x) = \chi_{\mathcal{D}}(x,t)$ for all $(x,t) \in \Omega \times (0,T)$.
\end{proof}
  The following result is a  consequence of  Corollary~\ref{yu-corollary-8-29-3}, which will be used to prove Theorem~\ref{yu-theorem-8-18-2}.

\begin{corollary}\label{yu-corollary-8-19-1}
     Suppose that  $T>0$ and $B^d_{4R}(x_0)\subset\Omega$ for some $R>0$ and $x_0\in\Omega$. Let $\mathcal{D}\subset B^d_R(x_0)\times(0,T)$ be a set of positive measure. With the notation of Lemma \ref{yu-lemma-8-19-1},  there exists a constant $C:=C(R,|\mathcal{D}|/T)>0$ such that for
      all $t\in E$,
\begin{equation}\label{yu-8-19-25}
\sum_{i=1}^{k_\lambda}|a_i|^2 \leq C e^{C\sqrt{\lambda}} \Big(\Big\|\chi_{\mathcal{D}_t}\sum_{i=1}^{k_\lambda} a_i e_i\Big\|_{L^1(\Omega;\mathbb{R})}\Big)^2\;\;\mbox{for all}\;\; \lambda> \lambda_1\;\;\mbox{and all}\;\;
\{a_i\}_{i=1}^{k_\lambda}\subset\mathbb{R},
\end{equation}
    where $\{\mathcal{D}_t\}_{t\in[0,T]}$ and $E$ are given by \eqref{yu-8-19-23}.
\end{corollary}
\begin{proof}
    It follows from \eqref{yu-8-19-23} that
    $\mathcal{D}_t\subset B_R^d(x_0)$ for each $t\in(0,T)$, and
    $|\mathcal{D}_t|\geq \frac{|\mathcal{D}|}{2T}$ for a.e. $t\in E$.
    Now we fix an arbitrary  $t\in E$. By   Corollary \ref{yu-corollary-8-29-3} with $\omega=\mathcal{D}_t$
     and $D=\frac{|\mathcal{D}|}{2T}$,
    we obtain \eqref{yu-8-19-25}. This completes  the proof.
\end{proof}

The following result is taken from \cite[Proposition 2.1]{PW} and will be used in the proof of
Theorem~\ref{yu-theorem-8-18-2}.

\begin{lemma}\label{a2}    Let $E\subset(0,T)$ be a measurable set of positive measure. Let $\ell$ be a density point of $E$. Then for each $\mu>1$, there exists an $\ell_1\in(\ell,T)$ such that the sequence $\{\ell_m\}_{m\geq1}$, given by
\begin{equation}\label{b1}
\ell_{m+1}=\ell+\frac{1}{\mu^m}(\ell_1-\ell),
\end{equation}
satisfies
$\ell_m-\ell_{m+1}\leq3|E\cap(\ell_{m+1},\ell_m)|$ for each $m\in\mathbb{N}^+$.
\end{lemma}
 We conclude this section with the following Remez-type inequality, which plays an important role in the
 proof of the integral-type interpolation observability inequality in Theorem~\ref{yu-theorem-8-18-1}--one of the key steps
  toward establishing
  Theorem~\ref{yu-theorem-8-18-2}.

\begin{lemma}\label{a4}
  For  $n\in\mathbb{N}^+$,  let $\mathcal{T}_n$ denote the space of all trigonometric polynomials of degree at most $n$ with real coefficients, i.e.,
\begin{equation*}
\mathcal{T}_n:=\Big\{f\in C([-\pi,\pi];\mathbb{R}):  f(\theta)=\sum_{k=0}^n\Big(a_k \sin (k \theta)+b_k \cos (k \theta)\Big), \theta \in(-\pi, \pi), \{(a_k,b_k)^\top\}_{k=0}^n \subset \mathbb{R}^2\Big\}.
\end{equation*}
Let $f\in\mathcal{T}_n$ and let $E\subset(-\pi, \pi)$ be a set of positive measure.  Then for every $p\in[1, \infty)$,
\begin{equation}\label{yu-8-22-7}
\|f\|_{L^p(-\pi, \pi;\mathbb{R})} \leq\Big(\frac{64}{\sin \frac{|E|}{4}}\Big)^{2(n+\frac{1}{p})}\|\chi_Ef\|_{L^p(-\pi, \pi;\mathbb{R})}.
\end{equation}
\end{lemma}
\begin{proof}
Fix $f\in\mathcal{T}_n$ arbitrarily. It follows from \cite[Theorem 1]{Ganzburg} that
for any set $\hat E \subset(-\pi, \pi)$ of positive measure,
\begin{equation}\label{yu-8-19-16}
    \|f\|_{C([-\pi,\pi];\mathbb{R})}\leq \Big(\frac{2}{\sin \frac{|\hat E|}{4}}\Big)^{2n}
    \sup_{t\in \hat E}|f(t)|.
\end{equation}
Let  $\varepsilon>0$. Define
\begin{equation}\label{yu-8-22-5-bbb}
    E_{f,\varepsilon}:=\Big\{\theta\in(-\pi,\pi):|f(\theta)|\leq \Big(\frac{1}{2}\sin\frac{\varepsilon}{4}\Big)^{2n}\|f\|_{C([-\pi,\pi];\mathbb{R})}\Big\}.
\end{equation}
   We claim  that
\begin{equation}\label{yu-4-2-1}
    |E_{f,\varepsilon}|\leq \varepsilon.
\end{equation}
  We prove it by contradiction.  Assume that \eqref{yu-4-2-1} does not hold, i.e.,  $|E_{f,\varepsilon}|>\varepsilon$. Then by \eqref{yu-8-19-16}, with $\hat E=E_{f,\varepsilon}$ (defined in  \eqref{yu-8-22-5-bbb}),  we have
\begin{eqnarray*}
    \|f\|_{C([-\pi,\pi];\mathbb{R})}&\leq& \Big(\frac{1}{2}\sin\frac{|E_{f,\varepsilon}|}{4}\Big)^{-2n}\sup_{\theta\in E_{f,\varepsilon}}|f(\theta)|\nonumber\\
    &\leq&\Big(\frac{1}{2}\sin\frac{|E_{f,\varepsilon}|}{4}\Big)^{-2n}
    \Big(\frac{1}{2}\sin\frac{\varepsilon}{4}\Big)^{2n}
    \|f\|_{C([-\pi,\pi];\mathbb{R})}
    <\|f\|_{C([-\pi,\pi];\mathbb{R})},
\end{eqnarray*}
which yields  a contradiction.

Now, we set  $\varepsilon:=\frac{1}{4}\sin\frac{|E|}{4}$. By \eqref{yu-4-2-1}, we have that $\Big|E_{f,\frac{1}{4}\sin\frac{|E|}{4}}\Big|\leq \frac{1}{4}\sin\frac{|E|}{4}$, which implies
$$
    \Big|E\cap \Big(E_{f,\frac{1}{4}\sin\frac{|E|}{4}}\Big)^c\Big|\geq |E|-\Big|E_{f,\frac{1}{4}\sin\frac{|E|}{4}}\Big|\geq |E|-\frac{1}{4}\sin\frac{|E|}{4}
    \geq \frac{1}{4}\sin\frac{|E|}{4}.
$$
  Since $\sin x-\frac{1}{2}x\geq 0$ for all $x\in[0,\frac{1}{16}]$, the above yields
\begin{eqnarray*}
    &\;&\|\chi_Ef\|_{L^p(-\pi,\pi;\mathbb{R})}^p\nonumber\\
    &\geq&\int_{-\pi}^\pi\chi_E(\theta)\chi_{\Big(E_{f,\frac{1}{4}\sin\frac{|E|}{4}}\Big)^c}(\theta)|f(\theta)|^pd\theta
    \geq\Big|E\cap \Big(E_{f,\frac{1}{4}\sin\frac{|E|}{4}}\Big)^c\Big|\Big(\frac{1}{2}
    \sin\Big(\frac{1}{16}\sin\frac{|E|}{4}\Big)\Big)^{2np}
    \|f\|^p_{C([-\pi,\pi];\mathbb{R})}\nonumber\\
    &\geq&\frac{1}{4}\sin\frac{|E|}{4}\Big(\frac{1}{64}\sin\frac{|E|}{4}\Big)^{2np}
    \|f\|^p_{C([-\pi,\pi];
    \mathbb{R})}\geq\Big(\frac{1}{64}\sin\frac{|E|}{4}\Big)^{2np+2}\int_{-\pi}^{\pi}|f(\theta)|^pd\theta.
\end{eqnarray*}
  Thus, inequality  \eqref{yu-8-22-7} holds. This completes the proof.
\end{proof} \color{black}

 \section{Integral-type interpolation observability inequality}\label{yu-sec-3}
    This section presents  a new integral-type interpolation observability  inequality, which plays an important  role in the proof of Theorem \ref{yu-theorem-8-18-2}.
    Let
\begin{equation}\label{yu-9-2-10}
    A:=\left(
        \begin{array}{cc}
          a & -b \\
          b & a \\
        \end{array}
      \right),\;\;\mathcal{A}:=\left(
                                 \begin{array}{cc}
                                   a\Delta & -b\Delta \\
                                   b\Delta & a\Delta \\
                                 \end{array}
                               \right)
\end{equation}
and
\begin{equation}\label{yu-9-2-11}
    B=(1,0).
\end{equation}
    It is clear that $\mathcal{A}$, with domain $H_0^1(\Omega;\mathbb{R}^2)\cap H^2(\Omega;\mathbb{R}^2)$, generates an
    analytic semigroup $\{e^{\mathcal{A}t}\}_{t\geq 0}$ on $L^2(\Omega;\mathbb{R}^2)$.
    For each $z\in L^2(\Omega;\mathbb{R}^2)$, let $\{v_j\}_{j\in\mathbb{N}^+}:=\{(v_{1,j},v_{2,j})^\top\}_{j\in\mathbb{N}^+}\subset\mathbb{R}^2$ denote the Fourier coefficients of $z$ with respect to the orthonormal basis $\{e_j\}_{j\in\mathbb{N}^+}$. Then
\begin{eqnarray}\label{yu-9-2-12}
    e^{\mathcal{A}t}z=\sum_{j=1}^{+\infty}e^{-A\lambda_j t} v_je_j
    =\sum_{j=1}^{+\infty}e^{-a\lambda_jt}\left(\begin{array}{rr}
\cos(\lambda_jbt) & \sin(\lambda_jbt) \\
-\sin(\lambda_jbt) & \cos(\lambda_jbt)
\end{array}\right)v_je_j.
\end{eqnarray}
\par
     The main result of this section is the following integral-type interpolation observability inequality.
\begin{theorem}\label{yu-theorem-8-18-1}
Let $E\subset (0,T)$ be a set of positive measure, and let
$\{\mathcal{D}_t\}_{t\in(0,T)}$ be a family of subsets of $\Omega$. Suppose that
\begin{enumerate}
\item[$(a)$] for a.e. $t\in (0,T)$, $\mathcal{D}_t$ is measurable;
\item[$(b)$] for any $f\in L^1(\Omega\times(0,T))$,
      the function $t\mapsto \int_{\Omega}\chi_{\mathcal{D}_t}(x)|f(x,t)|dx$ is integrable over $(0,T)$;
\item[$(c)$] there exists a constant $D>0$, independent of $E$, such that for a.e. $t\in E$,
\begin{equation*}
\sum_{i=1}^{k_\lambda}|a_i|^2
\leq D e^{D\sqrt{\lambda}} \Big(\|\chi_{\mathcal{D}_t}\textstyle\sum_{i=1}^{k_\lambda} a_i e_i\|_{L^1(\Omega;\mathbb{R})}\Big)^2
\;\;
\mbox{for all}\;\;\lambda>\lambda_1\;\;\mbox{and all}\;\;
\{a_i\}_{i=1}^{k_\lambda}\subset\mathbb{R}.
\end{equation*}
\end{enumerate}
Then there exists a constant $M=M(a,b,D)>0$ such that for all $\theta\in(0,1)$, all $0<S_1<S_2\leq T$ satisfying $|E\cap [S_1,S_2]|>0$, and all $z\in L^2(\Omega;\mathbb{R}^2)$,
\begin{equation}\label{yu-8-18-1}
\|e^{\mathcal{A}S_2}z\|_{L^2(\Omega;\mathbb{R}^2)}
\leq
K(E,\theta,S_1,S_2)\Big(\frac{1}{|E\cap [S_1,S_2]|}\int_{S_1}^{S_2}\chi_E(t)\| \chi_{\mathcal{D}_t} B e^{\mathcal{A}t}z\|_{L^1(\Omega;\mathbb{R})}dt\Big)^{1-\theta}\|z\|_{L^2(\Omega;\mathbb{R}^2)}^{\theta},
\end{equation}
where
\begin{equation*}
K(E,\theta,S_1,S_2):=\frac{M e^{M\big(\frac{S_2}{1-\theta}+\frac{1}{\theta S_1}\big)}}{|E\cap [S_1,S_2]|^{3}}.
\end{equation*}
\end{theorem}
\begin{proof}
First, we emphasize that, under Assumptions $(a)$ and $(b)$, the integral on the right-hand side of inequality \eqref{yu-8-18-1} is well defined for any $z\in L^2(\Omega;\mathbb{R}^2)$.

Our strategy is as follows: We first decompose the semigroup $\{e^{\mathcal{A}t}\}_{t\geq0}$
 into high- and low-frequency components. Next, we establish a low-frequency observability inequality for the observation operator $\chi_E\chi_{\mathcal{D}_t}B$ $(t\in (0,T))$,  with an explicit observability constant depending on the measure of the observation set.
 Finally,  we combine this inequality with the decay property of the high-frequency component to derive the desired result.
    The Remez-type inequality in Lemma~\ref{a4} plays a crucial role in establishing the low-frequency observability inequality.

The proof proceeds in four steps. First,  {\it we arbitrarily fix $S_1,S_2\in (0,T]$ with
$S_1<S_2$ and $|E\cap [S_1,S_2]|>0$, as well as $\theta \in(0,1)$.}  Set
      \begin{equation}\label{yu-8-20-10}
    E_{S_1,S_2}:=E\cap [S_1,S_2],\;\;\widehat{E}_{S_1,S_2}:=E_{S_1,S_2}-S_1\;\;\mbox{and}\;\;S:=S_2-S_1.
\end{equation}
    It is clear that $\widehat{E}_{S_1,S_2}\subset [0,S]$ and
\begin{equation}\label{yu-8-20-10-b}
    |\widehat{E}_{S_1,S_2}|=|E_{S_1,S_2}|>0.
\end{equation}
For each $\lambda>\lambda_1$, we define the following low-frequency  subspace:
$$
    H_{\lambda}:=\Big\{f=\sum_{i=1}^{k_\lambda}a_ie_i
    :\{a_i\}_{i=1}^{k_\lambda}\subset\mathbb{R}^2\Big\},
$$
where $k_\lambda$ is defined by \eqref{yu-8-19-00}.
\vskip 5pt
\noindent  \emph{Step 1: We prove that there exists a constant $N_1=N_1(b)>0$ such that for all $\lambda> \lambda_1$,
\begin{equation}\label{yu-8-20-11}
    \|e^{\mathcal{A}S}\varphi\|_{L^2(\Omega;\mathbb{R}^2)}\leq \frac{N_1(1+S)^4(1+\lambda)^{\frac{d}{4}}\lambda}{|\widehat{E}_{S_1,S_2}|^4}
    \int_{0}^{S}\chi_{\widehat{E}_{S_1,S_2}}(t)\|B e^{\mathcal{A}t}\varphi\|_{L^2(\Omega;\mathbb{R})}dt\;\;\mbox{for any}\;\;\varphi\in H_\lambda.
\end{equation}}
\par
   To this end, {\it we arbitrarily fix $\lambda>\lambda_1$ and $\varphi\in H_\lambda$.}
We write
    \begin{equation}\label{yu-2-26-1}\varphi:=\Big(
\sum_{j=1}^{k_\lambda}\varphi_{1,j}e_j,
\sum_{j=1}^{k_\lambda}\varphi_{2,j}e_j
\Big)^\top,
\end{equation}
where
    $\{(\varphi_{1,j},\varphi_{2,j})^\top\}_{j=1}^{k_\lambda}\subset \mathbb{R}^2$.
Then, by \eqref{yu-9-2-11} and \eqref{yu-9-2-12}, we have
\begin{equation}\label{a3}
B e^{\mathcal{A}t}\varphi=\sum_{j=1}^{k_\lambda}v_j(t)e_j,
\end{equation}
where
\begin{equation}\label{yu-8-9-1}
v_j(t):=e^{-a\lambda_j t}\Big(\varphi_{1,j}\cos(\lambda_jbt)
+\varphi_{2,j}\sin(\lambda_jbt)\Big),\quad t\in[0,T].
\end{equation}
Now,  {\it we arbitrarily fix $j\in\{1,2,\ldots,k_\lambda\}$.}
  By \eqref{yu-8-20-10-b} and \eqref{yu-8-9-1}, we obtain
\begin{eqnarray}\label{yu-8-11-1-bb}
    &\;&\int_0^S\chi_{\widehat{E}_{S_1,S_2}}(t)|v_j(t)|dt\nonumber\\
    &=&(\lambda_jb)^{-1}\int_0^{\lambda_jbS}\chi_{\widehat{E}^j_{S_1,S_2}}(\sigma)e^{-ab^{-1}\sigma}|\varphi_{1,j}\cos \sigma+\varphi_{2,j}\sin\sigma|d\sigma\nonumber\\
    &=&
\begin{cases}
    (\lambda_jb)^{-1}\int_0^{\lambda_jbS}\chi_{\widehat{E}^j_{S_1,S_2}}(\sigma)e^{-ab^{-1}\sigma}|\varphi_{1,j}\cos \sigma+\varphi_{2,j}\sin\sigma|d\sigma&\mbox{if}\;\;b>0,\\
    (\lambda_j|b|)^{-1}\int_0^{\lambda_j|b|S}\chi_{\widehat{E}^j_{S_1,S_2}}(\sigma)e^{-a|b|^{-1}\sigma}|\varphi_{1,j}\cos \sigma-\varphi_{2,j}\sin\sigma|d\sigma&\mbox{if}\;\;b<0,
\end{cases}
\end{eqnarray}
    where
$$\widehat{E}_{S_1,S_2}^j:= \Big\{\vartheta=\lambda_j|b|t:t\in \widehat{E}_{S_1,S_2}\Big\}.
$$
   We  use \eqref{yu-8-11-1-bb} to prove \eqref{yu-8-20-11}. First, we note that it suffices to consider
    \eqref{yu-8-11-1-bb} for the case  $b>0$.  Indeed, by
    \eqref{yu-9-2-12}, we have
   $$\|e^{\mathcal{A} t}(\varphi_1,\varphi_2)^\top\|_{L^2(\Omega;\mathbb{R}^2)}=\|e^{\mathcal{A} t}(\varphi_1,-\varphi_2)^\top\|_{L^2(\Omega;\mathbb{R}^2)}\;\;\mbox{for any}\;\;(\varphi_1,\varphi_2)^\top \in L^2(\Omega;\mathbb{R}^2).
$$
 Thus, via the transformation: $(\varphi_1,\varphi_2)^\top\rightarrow (\varphi_1,-\varphi_2)^\top$,  the case
 $b<0$ can be reduced to the case  $b>0$.
 Therefore, in what follows, we  assume  that $b>0$.
\par
    From \eqref{yu-8-11-1-bb}, we have
\begin{equation}\label{yu-8-11-1}
    \int_0^S\chi_{\widehat{E}_{S_1,S_2}}(t)|v_j(t)|dt\geq(\lambda_jb)^{-1}\big((\varphi_{1,j})^2+(\varphi_{2,j})^2\big)^{\frac{1}{2}}
    e^{-a\lambda_jS}\int_0^{\lambda_jbS}\chi_{\widehat{E}^j_{S_1,S_2}}(\sigma)|\sin(\sigma+\delta_j)|d\sigma,
\end{equation}
    where $\delta_j\in[-\frac{\pi}{2},\frac{\pi}{2}]$ is  defined by
\begin{equation*}\label{yu-8-11-2}
    \delta_j:=
\begin{cases}
   \mbox{arctan}\big(\varphi_{1,j}(\varphi_{2,j})^{-1}\big)&\mbox{if}\;\;\varphi_{2,j}\neq 0,\\
   \frac{\pi}{2}&\mbox{if}\;\;\varphi_{2,j}=0,\;\;\varphi_{1,j}>0,\\
    -\frac{\pi}{2}&\mbox{if}\;\;\varphi_{2,j}=0,\;\;\varphi_{1,j}<0,\\
    0&\mbox{if}\;\;\varphi_{1,j}=\varphi_{2,j}=0.
\end{cases}
\end{equation*}
We next estimate the integral on the right-hand side of \eqref{yu-8-11-1}.
Let $\widehat{F}^j_{S_1,S_2}:=\widehat{E}^j_{S_1,S_2}+\delta_j$. Then
\begin{equation}\label{yu-8-11-3}
    \int_0^{\lambda_jbS}\chi_{\widehat{E}^j_{S_1,S_2}}(\sigma)|\sin(\sigma+\delta_j)|d\sigma
    =\int_{\delta_j}^{\lambda_jbS+\delta_j}
    \chi_{\widehat{F}^j_{S_1,S_2}}(\eta)|\sin \eta|d\eta.
\end{equation}
\par
  A key step in estimating the integral on the right-hand side of \eqref{yu-8-11-1} is to prove that
\begin{equation}\label{yu-8-11-4}
    2^{-50}\Big(\lambda_jbS+\frac{\pi}{2}\Big)^{-4}|\widehat{F}^j_{S_1,S_2}|^4\leq
     \int_{\delta_j}^{\lambda_jbS+\delta_j}
    \chi_{\widehat{F}^j_{S_1,S_2}}(\eta)|\sin \eta|d\eta.
\end{equation}
This is achieved by considering the following two cases:

   \vskip 5pt
    \emph{In the case that
    $\lambda_jbS+\delta_j\leq \pi$,}
      we use the fact that $\delta_j\in[-\frac{\pi}{2},\frac{\pi}{2}]$ to obtain
  \begin{equation}\label{yu-8-12-1}
    \int_{\delta_j}^{\lambda_jbS+\delta_j}
    \chi_{\widehat{F}^j_{S_1,S_2}}(\eta)|\sin \eta|d\eta=\int_{-\pi}^{\pi} \chi_{\widehat{F}^j_{S_1,S_2}}(\eta)|\sin \eta|d\eta.
\end{equation}
   Let
\begin{equation}\label{yu-8-12-1-bb}
    I_i:=\Big[\big(-1+\frac{i}{4}\Big)\pi,\Big(-1+\frac{i+1}{4}\Big)\pi\Big],\;\;i\in\{0,1,\ldots, 7\}.
\end{equation}
    Then, $[-\pi,\pi]=\cup_{i=0}^7 I_i$. Choose
        $i^*\in \{0,1,\ldots, 7\}$ such that
$|\widehat{F}^j_{S_1,S_2}\cap I_{i^*}|=\max_{0\leq i\leq 7}|\widehat{F}^j_{S_1,S_2}\cap I_{i}|$.
      It is clear that
\begin{equation}\label{yu-8-12-2}
    |\widehat{F}^j_{S_1,S_2}\cap I_{i^*}|<\frac{\pi}{3}\;\;\mbox{and}\;\;|\widehat{F}^j_{S_1,S_2}\cap I_{i^*}|\geq \frac{1}{8}|\widehat{F}^j_{S_1,S_2}|.
\end{equation}
  Since  $\int_{-\pi}^{\pi} |\sin \eta|d\eta=4$ and  $\sin(\frac{x}{4})-\frac{x}{16}>0$ for each $x\in (0,\frac{\pi}{3})$, it follows from \eqref{yu-8-12-1},
      Lemma \ref{a4} and \eqref{yu-8-12-2} that
\begin{eqnarray*}\label{yu-8-12-3}
     \int_{\delta_j}^{\lambda_jbS+\delta_j}
    \chi_{\widehat{F}^j_{S_1,S_2}}(\eta)|\sin \eta|d\eta
    &=& \sum_{i=0}^7\int_{-\pi}^{\pi}\chi_{\widehat{F}^j_{S_1,S_2}\cap I_{i}}(\eta)|\sin \eta|d\eta\geq \int_{-\pi}^{\pi} \chi_{\widehat{F}^j_{S_1,S_2}\cap I_{i^*}}(\eta)|\sin \eta|d\eta\nonumber\\
    &\geq& 2^{-24}\Big(\sin\Big(\frac{1}{4}|\widehat{F}^j_{S_1,S_2}\cap I_{i^*}|\Big)\Big)^4\int_{-\pi}^{\pi} |\sin \eta|d\eta
    \geq 2^{-50}|\widehat{F}^j_{S_1,S_2}|^4
   \nonumber\\
    &\geq&2^{-50}\Big(\lambda_jbS+\frac{\pi}{2}\Big)^{-4}|\widehat{F}^j_{S_1,S_2}|^4.
\end{eqnarray*}
   Hence, \eqref{yu-8-11-4} holds in this case.

  \vskip 5pt
\emph{In the case that
    $\lambda_jbS+\delta_j> \pi$,} we set
$$
    m_j:=\Big[\frac{\lambda_jbS+\delta_j-\pi}{2\pi}\Big]_{\mathbb{Z}}.
$$
    Then
    $m_j\geq 0$, and
    \begin{equation}\label{yu-8-13-1}
    (2m_j+1)\pi\leq \lambda_jbS+\delta_j<(2m_j+3)\pi.
\end{equation}
    Moreover,
\begin{equation}\label{yu-8-13-2}
    \int_{\delta_j}^{\lambda_jbS+\delta_j}
    \chi_{\widehat{F}^j_{S_1,S_2}}(\eta)|\sin \eta|d\eta=\sum_{k=-1}^{m_j}\int_{(2k+1)\pi}^{(2k+3)\pi}
    \chi_{\widehat{F}^{j,k}_{S_1,S_2}}(\eta)|\sin\eta| d\eta,
\end{equation}
    where
$$
    \widehat{F}^{j,k}_{S_1,S_2}:=\widehat{F}^j_{S_1,S_2}\cap [(2k+1)\pi, (2k+3)\pi],\;\;k\in\{-1,0,\ldots,m_j\}.
$$
  Choose $k^*\in\{-1,0,\ldots,m_j\}$ such that
$|\widehat{F}^{j,k^*}_{S_1,S_2}|=\max_{-1\leq k\leq m_j}|\widehat{F}^{j,k}_{S_1,S_2}|$. Then
    \begin{equation}\label{yu-8-13-4}
     |\widehat{F}^{j,k^*}_{S_1,S_2}|\geq (m_j+2)^{-1}|\widehat{F}^j_{S_1,S_2}|.
\end{equation}
   We now  claim that
\begin{equation}\label{yu-8-13-5}
\int_{(2k^*+1)\pi}^{(2k^*+3)\pi}
    \chi_{\widehat{F}^{j,k^*}_{S_1,S_2}}(\eta)|\sin\eta| d\eta
    \geq 2^{-50}|\widehat{F}^{j,k^*}_{S_1,S_2}|^4.
\end{equation}
    To this end, we define
$$
    G^{j, k^*}:=\widehat{F}^{j,k^*}_{S_1,S_2}-2(k^*+1)\pi.
$$
      Then $|G^{j,k^*}|=|\widehat{F}^{j,k^*}_{S_1,S_2}|$,
     $G^{j,k^*}\subset[-\pi,\pi]$, and
\begin{equation}\label{yu-8-13-9}
    \int_{(2k^*+1)\pi}^{(2k^*+3)\pi}
    \chi_{\widehat{F}^{j,k^*}_{S_1,S_2}}(\eta)|\sin\eta| d\eta=\int_{-\pi}^{\pi}
    \chi_{G^{j,k^*}}(\eta)|\sin\eta| d\eta.
\end{equation}
    Let
    $i^*\in\{0,\ldots,7\}$ be such that
$|G^{j,k^*}\cap I_{i^*}|
    =\max_{0\leq i\leq 7}|G^{j,k^*}\cap I_{i}|$,
     where $I_i$ for $i\in\{0,\ldots,7\}$ is given in \eqref{yu-8-12-1-bb}. Then
\begin{equation}\label{yu-8-13-7}
    |G^{j,k^*}\cap I_{i^*}|<\frac{\pi}{3}\;\;\mbox{and}\;\;|G^{j,k^*}\cap I_{i^*}|\geq \frac{1}{8}|G^{j,k^*}|.
\end{equation}
    It follows from  \eqref{yu-8-13-9},
     Lemma \ref{a4} and \eqref{yu-8-13-7} that
\begin{eqnarray*}\label{yu-8-13-8}
    \int_{(2k^*+1)\pi}^{(2k^*+3)\pi}
    \chi_{\widehat{F}^{j,k^*}_{S_1,S_2}}(\eta)|\sin\eta| d\eta
    &=&
    \int_{-\pi}^{\pi}\chi_{G^{j,k^*}\cap I_{i^*}}(\eta)|\sin\eta|d\eta\geq \nonumber\\
    &\geq& 2^{-22}\Big(\sin\Big(\frac{1}{4}|G^{j,k^*}\cap I_{i^*}|\Big)\Big)^4
    \geq 2^{-50}|G^{j,k^*}|^4=2^{-50}|\widehat{F}^{j,k^*}_{S_1,S_2}|^4,
\end{eqnarray*}
i.e.,     \eqref{yu-8-13-5} holds.

 Now, it follows from \eqref{yu-8-13-1}, \eqref{yu-8-13-2}, \eqref{yu-8-13-4}
    and \eqref{yu-8-13-5} that
\begin{eqnarray*}\label{yu-8-13-11}
     \int_{\delta_j}^{\lambda_jbS+\delta_j}
    \chi_{\widehat{F}^j_{S_1,S_2}}(\eta)|\sin \eta|d\eta\geq \int_{(2k^*+1)\pi}^{(2k^*+3)\pi}
    \chi_{\widehat{F}^{j,k^*}_{S_1,S_2}}(\eta)|\sin\eta| d\eta
    \geq 2^{-50}(m_j+2)^{-4}|\widehat{F}^j_{S_1,S_2}|^4\nonumber\\
    \geq 2^{-50}
    \Big(\lambda_jbS+\frac{\pi}{2}\Big)^{-4}|\widehat{F}^j_{S_1,S_2}|^4,
\end{eqnarray*}
  which yields  \eqref{yu-8-11-4} in this case. Therefore, combining the two cases, \eqref{yu-8-11-4} holds.
\par
   Meanwhile, it follows from the definitions of $\widehat{F}^j_{S_1,S_2}$ and $\widehat{E}^j_{S_1,S_2}$ that
$$|\widehat{F}^j_{S_1,S_2}|=|\widehat{E}^j_{S_1,S_2}|=\lambda_jb|\widehat{E}_{S_1,S_2}|.$$
    Hence,
\begin{equation*}\label{yu-8-13-13}
    \Big(\lambda_jbS+\frac{\pi}{2}\Big)^{-1}|\widehat{F}^j_{S_1,S_2}|=
    \Big(S+\frac{\pi}{2\lambda_jb}\Big)^{-1}|\widehat{E}_{S_1,S_2}|
    \geq \Big(S+\frac{\pi}{2\lambda_1b}\Big)^{-1}|\widehat{E}_{S_1,S_2}|.
\end{equation*}
    This, together with \eqref{yu-8-11-4}, implies
\begin{equation}\label{yu-8-13-14}
    2^{-50} \Big(S+\frac{\pi}{2\lambda_1b}\Big)^{-4}|\widehat{E}_{S_1,S_2}|^4\leq \int_{\delta_j}^{\lambda_jbS+\delta_j}
    \chi_{\widehat{F}^j_{S_1,S_2}}(\eta)|\sin\eta|d\eta.
\end{equation}

    Now,  by \eqref{yu-9-2-12}, \eqref{yu-8-11-1}, \eqref{yu-8-11-3} and \eqref{yu-8-13-14}, we have
\begin{eqnarray}\label{yu-8-13-15}
    &\;&k_\lambda\Big(\int_0^S\chi_{\widehat{E}_{S_1,S_2}}(t)\Big(\sum_{j=1}^{k_\lambda}|v_j(t)|^2\Big)^{\frac{1}{2}}
    dt\Big)^2\nonumber\\
    &\geq& \sum_{j=1}^{k_\lambda}\Big(\int_0^S\chi_{\widehat{E}_{S_1,S_2}}(t)|v_j(t)|dt\Big)^2\nonumber\\
    &\geq&2^{-100}(\lambda b)^{-2}\Big(S+\frac{\pi}{2\lambda_1b}\Big)^{-8}|\widehat{E}_{S_1,S_2}|^8
    \sum_{j=1}^{k_\lambda}((\varphi_{1,j})^2+(\varphi_{2,j})^2)e^{-2\lambda_jaS}\nonumber\\
    &=&2^{-100}(\lambda b)^{-2}\Big(S+\frac{\pi}{2\lambda_1b}\Big)^{-8}|\widehat{E}_{S_1,S_2}|^8
    \|e^{\mathcal{A}S}\varphi\|^2_{L^2(\Omega;\mathbb{R}^2)}.
\end{eqnarray}
However, it follows from \eqref{yu-nn-1-1} that there exists
a constant  $C_1>0$ (which only depends on $d$ and $\Omega$) such that
$k_\lambda\leq C_1(1+\lambda^{\frac{d}{2}})$ for all $\lambda>0$.
    This, together with \eqref{yu-8-13-15}, yields  a  constant $C_2=C_2(b)>0$ such that
\begin{eqnarray*}\label{yu-8-14-2}
    \|e^{\mathcal{A}S}\varphi\|_{L^2(\Omega;\mathbb{R}^2)}\leq \frac{C_2(1+S)^4(1+\lambda)^{\frac{d}{4}}\lambda}{|\widehat{E}_{S_1,S_2}|^4}
    \int_0^S\chi_{\widehat{E}_{S_1,S_2}}(t)\Big(\sum_{j=1}^{k_\lambda}|v_j(t)|^2\Big)^{\frac{1}{2}}dt.
\end{eqnarray*}
 The above, combined  with  \eqref{a3}, yields \eqref{yu-8-20-11} with $N_1=C_2$.

\vskip 5pt
 \noindent  \emph{Step 2: We show that there exists a constant $N_2=N_2(b,D)>0$ such that for all $\lambda> \lambda_1$,
\begin{equation}\label{yu-8-14-00}
\|e^{\mathcal{A}S_2}\varphi\|_{L^2(\Omega;\mathbb{R}^2)}
\leq \frac{N_2(1+S_2-S_1)^4\lambda e^{N_2\sqrt{\lambda}}}{|E_{S_1,S_2}|^4}
\int_{S_1}^{S_2}\chi_E(t)\| \chi_{\mathcal{D}_t} B e^{\mathcal{A}t}\varphi\|_{L^1(\Omega;\mathbb{R})}dt
\;\;\mbox{for all}\;\;\varphi\in H_\lambda.
\end{equation}}

{\it We arbitrarily fix $\lambda> \lambda_1$ and $\varphi\in H_\lambda$, which is given by \eqref{yu-2-26-1}. Therefore,  by \eqref{yu-9-2-11} and \eqref{yu-9-2-12}, we have \eqref{a3}.}
\par
    We first show that
\begin{equation}\label{yu-8-9-3}
    \int_{S_1}^{S_2}\chi_E(t)\|B e^{\mathcal{A}t}\varphi\|_{L^2(\Omega;\mathbb{R})}dt
    \leq De^{D\sqrt{\lambda}}\int_{S_1}^{S_2}\chi_E(t)\|\chi_{\mathcal{D}_t} B e^{\mathcal{A}t}\varphi\|_{L^1(\Omega;\mathbb{R})}dt.
\end{equation}
Indeed, by Assumption $(c)$ and \eqref{a3}, we have that for a.e. $t\in E$,
\begin{eqnarray*}\label{yu-8-9-2}
&\;&\chi_{E}(t)\|B e^{\mathcal{A}t}\varphi\|^2_{L^2(\Omega;\mathbb{R})}=\chi_{E}(t)
\Big(\sum_{j=1}^{k_\lambda}|v_j(t)|^2\Big)\nonumber\\
&\leq& D^2e^{2D\sqrt{\lambda}}\chi_{E}(t)\Big\|\chi_{\mathcal{D}_t}
\sum_{j=1}^{k_\lambda}v_j(t)e_j\Big\|^2_{L^1(\Omega;\mathbb{R})}
=D^2e^{2D\sqrt{\lambda}}\chi_{E}(t)\|\chi_{\mathcal{D}_t} B e^{\mathcal{A}t}\varphi\|^2_{L^1(\Omega;\mathbb{R})},
\end{eqnarray*}
 which yields \eqref{yu-8-9-3}.
\par
    We now prove \eqref{yu-8-14-00}.
  Replacing $\varphi$  with $e^{\mathcal{A}S_1}\varphi$ in  \eqref{yu-8-20-11}, we obtain
$$
    \|e^{\mathcal{A}(S+S_1)}\varphi\|_{L^2(\Omega;\mathbb{R}^2)}\leq \frac{N_1(1+S)^4(1+\lambda)^{\frac{d}{4}}\lambda}{|\widehat{E}_{S_1,S_2}|^4}
    \int_{0}^{S}\chi_{\widehat{E}_{S_1,S_2}}(t)\|B e^{\mathcal{A}(t+S_1)}\varphi\|_{L^2(\Omega;\mathbb{R})}dt.
$$
 Combining this with \eqref{yu-8-20-10}, \eqref{yu-8-20-10-b} and the change of variable  $t\to t-S_1$, we get
\begin{equation}\label{yu-8-9-4}
   \|e^{\mathcal{A}S_2}\varphi\|_{L^2(\Omega;\mathbb{R}^2)}\leq \frac{N_1(1+S_2-S_1)^4(1+\lambda)^{\frac{d}{4}}\lambda}{|E_{S_1,S_2}|^4}
    \int_{S_1}^{S_2}\chi_E(t)\|B e^{\mathcal{A}t}\varphi\|_{L^2(\Omega;\mathbb{R})}dt.
\end{equation}
Now, \eqref{yu-8-14-00} follows immediately from
 \eqref{yu-8-9-3} and \eqref{yu-8-9-4} with $N_2:=(N_1+1)D$.

\vskip 5pt
\noindent \emph{Step 3: We show that there exist two positive constants $C_3=C_3(b,D)$ and $C_4=C_4(a,b,D)$
such that for all $\lambda> \lambda_1$ and $z\in L^2(\Omega;\mathbb{R}^2)$,
\begin{eqnarray}\label{yu-8-15-8}
    \|e^{\mathcal{A}S_2}z\|_{L^2(\Omega;\mathbb{R}^2)}
     \leq \frac{N_3}{|E_{S_1,S_2}|^3}
   \Big[\frac{1}{\varepsilon(\lambda)^\gamma}\Big(\frac{1}{|E_{S_1,S_2}|}\int_{S_1}^{S_2}
    \chi_E(t)\|\chi_{\mathcal{D}_t} B e^{\mathcal{A}t}z\|_{L^1(\Omega;\mathbb{R})}dt\Big)
    +\varepsilon(\lambda)\|z\|_{L^2(\Omega;\mathbb{R}^2)}\Big],
\end{eqnarray}
    where $\varepsilon(\lambda):=e^{-\frac{1}{2}a(\lambda-\lambda_1) S_1}$,  $\gamma=\theta(1-\theta)^{-1}$,
    and
    \begin{eqnarray}\label{yu-2-26-2}
    N_3:= C_4(1+S_2)^4e^{\big(\frac{a\lambda_1S_2}{2(1-\theta)}+\frac{(C_3)^2}{a\theta S_1}\big)}.
    \end{eqnarray}}

 To this end, {\it we arbitrarily fix  $z\in L^2(\Omega;\mathbb{R}^2)$ and $\lambda>\lambda_1$.}
    We write
\begin{equation}\label{yu-8-14-3}
   z=z^1+z^2:=\sum_{j=1}^{k_\lambda}z_je_j +\sum_{j=k_\lambda+1}^{+\infty}z_je_j,
\end{equation}
    where $\{z_j\}_{j\in\mathbb{N}^+}\subset\mathbb{R}^2$ are the Fourier coefficients of $z$.
   We state two facts: First,  by
       \eqref{yu-9-2-12},
\begin{equation*}\label{yu-8-15-1}
   \|e^{\mathcal{A}t}z^2\|_{L^2(\Omega;\mathbb{R}^2)}\leq e^{-a\lambda t}\|z^2\|_{L^2(\Omega;\mathbb{R}^2)} \;\;\mbox{for all}\;\;t\in[0,T];
\end{equation*}
    Second, applying  \eqref{yu-8-14-00} to $z^1$ (i.e., letting $\varphi=z^1$  in \eqref{yu-8-14-00}) gives
\begin{eqnarray*}\label{yu-8-15-2}
    \|e^{\mathcal{A}S_2}z^1\|_{L^2(\Omega;\mathbb{R}^2)}
       \leq\frac{N_2(1+S_2-S_1)^4\lambda e^{N_2\sqrt{\lambda}}}{|E_{S_1,S_2}|^4}\int_{S_1}^{S_2}\chi_E(t)\|\chi_{\mathcal{D}_t} B e^{\mathcal{A}t}z^1\|_{L^1(\Omega;\mathbb{R})}dt.
\end{eqnarray*}
   These facts, together with \eqref{yu-8-14-3}, imply
\begin{eqnarray}\label{yu-8-15-3}
    \|e^{\mathcal{A}S_2}z\|_{L^2(\Omega;\mathbb{R}^2)}
    &\leq&\|e^{\mathcal{A}S_2}z^1\|_{L^2(\Omega;\mathbb{R}^2)}+e^{-a\lambda S_2}\|z^2\|_{L^2(\Omega;\mathbb{R}^2)}\nonumber\\
    &\leq&\frac{N_2(1+S_2-S_1)^4\lambda e^{N_2\sqrt{\lambda}}}{|E_{S_1,S_2}|^4}\biggl(\int_{S_1}^{S_2}\chi_E(t)\|\chi_{\mathcal{D}_t} B e^{\mathcal{A}t}z\|_{L^1(\Omega;\mathbb{R})}dt\nonumber\\
    &\;&+\int_{S_1}^{S_2}\chi_E(t)\|\chi_{\mathcal{D}_t} B e^{\mathcal{A}t}z^2\|_{L^1(\Omega;\mathbb{R})}dt\biggl)+e^{-a\lambda S_2}\|z^2\|_{L^2(\Omega;\mathbb{R}^2)}\nonumber\\
    &\leq& \frac{N_2(1+S_2-S_1)^4\lambda e^{N_2\sqrt{\lambda} }}{|E_{S_1,S_2}|^4}\int_{S_1}^{S_2}\chi_E(t)\|\chi_{\mathcal{D}_t} B e^{\mathcal{A}t}z\|_{L^1(\Omega;\mathbb{R})}dt  \nonumber\\
    &\;&  +\Big(\frac{N_2|\Omega|^{\frac{1}{2}}(1+S_2-S_1)^4\lambda e^{N_2\sqrt{\lambda}-a\lambda S_1}}{|E_{S_1,S_2}|^3}
    +e^{-a\lambda S_2}\Big)\|z\|_{L^2(\Omega;\mathbb{R}^2)}.
\end{eqnarray}
  However, it is easy to verify that there exists a constant $C_3=C_3(b,D)>0$ such that
\begin{equation*}\label{yu-8-15-4}
\begin{cases}
    N_2\lambda e^{N_2\sqrt{\lambda}}\leq C_3e^{C_3\sqrt{\lambda}},\\
    N_2|\Omega|^{\frac{1}{2}}(1+S_2-S_1)^4\lambda e^{N_2\sqrt{\lambda}-a\lambda S_1}+|E_{S_1,S_2}|^3e^{-a\lambda S_2}
    \leq
    C_3(1+S_2)^4e^{C_3\sqrt{\lambda}-a\lambda S_1}.
\end{cases}
\end{equation*}
  This, along with  \eqref{yu-8-15-3}, yields
\begin{eqnarray}\label{yu-8-15-5}
    &\;&\|e^{\mathcal{A}S_2}z\|_{L^2(\Omega;\mathbb{R}^2)}\nonumber\\
    &\leq& \frac{C_3(1+S_2)^4}{|E_{S_1,S_2}|^3}
\Big(\frac{e^{C_3\sqrt{\lambda}}}{|E_{S_1,S_2}|}\int_{S_1}^{S_2}\chi_E(t)\|\chi_{\mathcal{D}_t} B e^{\mathcal{A}t}z\|_{L^1(\Omega;\mathbb{R})}dt
    +e^{C_3\sqrt{\lambda}-a\lambda S_1}\|z\|_{L^2(\Omega;\mathbb{R}^2)}\Big).
\end{eqnarray}
   Meanwhile, since  $\gamma=\theta(1-\theta)^{-1}$,  after some elementary calculations, we have
\begin{equation*}\label{yu-8-15-7}
    C_3\sqrt{\lambda}-\frac{1}{2}a\lambda S_1\leq \frac{(C_3)^2}{aS_1}\leq \frac{(C_3)^2}{a\theta S_1}
    \;\;\mbox{and}\;\;C_3\sqrt{\lambda}\leq \frac{1}{2}a\gamma \lambda S_1+\frac{(C_3)^2}{2a\gamma S_1}
    \leq \frac{1}{2}a\gamma \lambda S_1+\frac{(C_3)^2}{a\theta S_1}.
\end{equation*}
   Combining these estimates with  \eqref{yu-8-15-5} yields \eqref{yu-8-15-8} for some constant $C_4=C_4(a,b,D)>0$.

    \vskip 5pt
   \noindent \emph{Step 4: We  show  inequality \eqref{yu-8-18-1}.}

   We now use Lemma \ref{yu-lemma-9-4-3} to prove inequality \eqref{yu-8-18-1},
where $$H=L^2(\Omega;\mathbb{R}^2),\quad F_1(\cdot)=\|e^{\mathcal{A}S_2}\cdot\|_{L^2(\Omega;\mathbb{R}^2)},$$
$$
F_2(\cdot)=\frac{1}{|E_{S_1,S_2}|}\int_{S_1}^{S_2}\chi_E(t)\|\chi_{\mathcal{D}_t} B e^{\mathcal{A}t}\cdot\|_{L^1(\Omega;\mathbb{R})}dt\;\;\mbox{and}\;\;
 F_3(\cdot)=\|\cdot\|_{L^2(\Omega;\mathbb{R}^2)}.$$
\par
Since the range of the function $\lambda\mapsto \varepsilon(\lambda):=e^{-\frac{1}{2}a(\lambda-\lambda_1) S_1}$ for $\lambda>\lambda_1$ is $(0,1)$, it follows from \eqref{yu-8-15-8} that the statement $(i)$ in Lemma \ref{yu-lemma-9-4-3} holds with
     $\Pi_1=N_3/|E_{S_1,S_2}|^3$.
By the equivalence stated in Lemma \ref{yu-lemma-9-4-3}, we have that for all $z\in L^2(\Omega;\mathbb{R}^2)$,
$$
    \|e^{\mathcal{A}S_2}z\|_{L^2(\Omega;\mathbb{R}^2)}
\leq\frac{2N_3}{|E_{S_1,S_2}|^3}\Big(\frac{1}{|E\cap [S_1,S_2]|}\int_{S_1}^{S_2}\chi_E(t)\| \chi_{\mathcal{D}_t} B e^{\mathcal{A}t}z\|_{L^1(\Omega;\mathbb{R})}dt\Big)^{1-\theta}\|z\|_{L^2(\Omega;\mathbb{R}^2)}^{\theta}.
$$
  From \eqref{yu-2-26-2},  there exists a positive constant $\widehat{M}$ such that
   $2N_3\leq \widehat{M} e^{\widehat{M}\left(\frac{S_2}{1-\theta}+\frac{1}{\theta S_1}\right)}$.
   This, along with the definition of $E_{S_1,S_2}$ (see \eqref{yu-8-20-10}), yields inequality \eqref{yu-8-18-1} with $M=\widehat{M}$.
   We complete the proof.
       \end{proof}

\begin{remark}\label{yu-remark-9-2-1}
     Theorem \ref{yu-theorem-8-18-1} establishes an integral-type interpolation observability inequality for the strongly coupled parabolic system \eqref{yu-8-28-1} with single-component observation over measurable sets.
For weakly coupled parabolic systems, inequalities of this type over open subsets have been established in \cite[Lemma 2.6]{Qin-Wang-Yu}.
By contrast, to the best of our knowledge, the validity of such inequalities for strongly coupled parabolic systems has not been addressed in the existing literature.
Against this background, the following two issues are of particular interest:
   \begin{enumerate}
  \item [$(P1)$] Does the following pointwise-in-time interpolation observability inequality hold?
\begin{equation}\label{yu-9-2-1}
    \|e^{\mathcal{A}T}z\|_{L^2(\Omega;\mathbb{R}^2)}\leq C
    \left(\sum_{i=1}^m\|\chi_{\omega_i}B e^{\mathcal{A}S_i}z\|_{L^1(\Omega;\mathbb{R})}\right)^{1-\theta}
    \|z\|^\theta_{L^2(\Omega;\mathbb{R}^2)}\;\;\mbox{for all}\;\;z\in L^2(\Omega;\mathbb{R}^2),
\end{equation}
      where $m\in\mathbb{N}^+$,  $\omega_i$ ($i=1,2,\ldots,m$) are measurable subsets of $\Omega$ with positive measure, $\{S_i\}_{i=1}^m\subset(0,T)$, $\theta\in(0,1)$, and $C>0$.
\par
    When $m=1$, inequality \eqref{yu-9-2-1} is referred to as the single-time-point interpolation observability inequality; when $m>1$, it is referred to as the multi-time-point interpolation observability inequality.
\par
    When $\omega_i=\omega$ for each $i\in\{1,2,\ldots,m\}$, where $\omega$ is an open subset of $\Omega$, the $L^2$-type multi-time-point interpolation observability inequality (i.e., the observation term is measured in the $L^2$-norm) for weakly coupled parabolic systems was  established in \cite[Proposition 1.3 and Lemma 5.1]{Wang-Yan-Yu}, under the assumptions that the pair formed by the coupling and observation matrices satisfies the Kalman rank condition and that sufficiently many observation time points are available. Therefore, it remains of interest to investigate whether such an inequality holds for strongly coupled systems.
 \item [$(P2)$] Do interpolation observability inequalities, either pointwise-in-time or of integral-type, remain valid for the following observation operators?
  \begin{equation}\label{yu-9-2-2}
    \widehat{B}:=(\mu_1,\mu_2),\;\mbox{with}\;\mu_1,\mu_2\in\mathbb{R}\;\mbox{and}\;|\mu_1|+|\mu_2|\neq 0;
      \;\;\mbox{and}\;\;\widetilde{B}:=I:=\mbox{diag}\{1,1\},
\end{equation}
  where $\widehat{B}$ corresponds to observation along the direction $(\mu_1,\mu_2)$ in $\mathbb{R}^2$,
   while $\widetilde{B}$ corresponds to observation along directions $(1,0)$ and $(0,1)$ in $\mathbb{R}^2$.
         These cases are closely related
     to the  null controllability of the controlled parabolic equation with complex coefficients from  measurable set:
\begin{equation}\label{yu-4-3-1}
\begin{cases}
    w_t=(a+\mathrm{i}b)\Delta_{\mathbb{C}} w+\chi_{\mathcal{D}}u&\mbox{in}\;\;\Omega\times(0,T),\\
    w=0&\mbox{on}\;\;\partial\Omega\times (0,T),\\
    w(0)=w_0&\mbox{in}\;\;\Omega,
\end{cases}
\end{equation}
     where $u\in L^\infty(0,T;L^2(\Omega;\mathbb{C}))$, $w_0\in L^2(\Omega;\mathbb{C})$, $\mathcal{D}$ is
     a measurable subset of $\Omega\times(0,T)$ with $|\mathcal{D}|>0$ and $\chi_{\mathcal{D}}$ is the characteristic function of $\mathcal{D}$.
      Indeed, the null controllability of system \eqref{yu-4-3-1} is equivalent to the observability (with observation set $\Omega$) of system
      \eqref{yu-8-20-1} (see, for instance,  \cite[Lemma 5.1]{Wang-Wang-Zhang} for the abstract setting).
    The choice  $B=(1,0)$ corresponds to the null controllability of \eqref{yu-4-3-1} by a real-valued control acting on its real part.  The choice $\widehat{B}$ with $\mu_2\neq 0$ represents a  real-valued control acting through a genuinely complex direction. The case $\widetilde{B}$ corresponds to a fully complex-valued control acting on both components.

\end{enumerate}
\end{remark}

The following propositions address the issues posed in Remark \ref{yu-remark-9-2-1}. The first proposition shows that the single-time-point interpolation observability inequality (i.e., inequality \eqref{yu-9-2-1} with $m=1$) does not hold.

\begin{proposition}\label{yu-proposition-9-2-1}
      For each $S\in(0,T)$, there exists
    $z\in L^2(\Omega;\mathbb{R}^2)\setminus\{0\}$ such that
\begin{equation}\label{yu-9-2-3}
    Be^{\mathcal{A}S}z=0.
\end{equation}
\end{proposition}
\begin{proof}
    We arbitrarily fix  $S\in (0,T)$ and $j\in \mathbb{N}^+$.
  We define $(v_1,v_2)^\top \in \mathbb{R}^2$ by
\begin{equation*}\label{yu-4-8-4}
    (v_1,v_2):=
    (-\sin (\lambda_jbS),\cos(\lambda_jbS)).
\end{equation*}
  Then   $(v_1)^2+(v_2)^2=1$, and
     \begin{equation}\label{yu-1-29-3}
   v_1\cos(\lambda_jbS)+v_2\sin(\lambda_jbS)=0.
\end{equation}
    Now, set $z:=(v_1,v_2)^\top e_j\in L^2(\Omega;\mathbb{R}^2)$. By \eqref{yu-9-2-12},  we have
\begin{eqnarray*}\label{yu-9-2-13}
    Be^{\mathcal{A}S}z= e^{-a\lambda_j S}(v_1\cos(\lambda_j bS)+v_2\sin(\lambda_jbS))e_j.
\end{eqnarray*}
   Combining this with \eqref{yu-1-29-3}, we obtain   \eqref{yu-9-2-3}.
    This completes the proof.
\end{proof}
The second proposition indicates that, regardless of the number of observation points, the multi-time-point interpolation observability inequality (i.e., inequality \eqref{yu-9-2-1} with $m>1$) may fail to hold.

\begin{proposition}\label{yu-proposition-4-8-1}
   For each $m\in \mathbb{N}^+$, there exist $\{S_i\}_{i=1}^m\subset(0,T)$  and
     $z\in  L^2(\Omega;\mathbb{R}^2)\setminus\{0\}$ such that
\begin{equation}\label{yu-4-8-1}
    Be^{\mathcal{A}S_i}z=0\;\;\mbox{for all}\;\;i\in\{1,2,\ldots,m\}.
\end{equation}
\end{proposition}
\begin{proof}
Fix $m\in \mathbb{N}^+$ arbitrarily. Since $\lambda_n\to +\infty$ as $n\to\infty$, we can choose $n\in\mathbb{N}^+$ sufficiently large such that
\begin{equation}\label{yu-4-8-2-new}
    \frac{2\pi}{|b|\lambda_n}\le \frac{T}{m+1}.
\end{equation}
Define
\begin{equation}\label{yu-4-8-3-new}
S_i:=
\begin{cases}
\dfrac{2i\pi}{b\lambda_n} & \mbox{if } b>0,\\[2mm]
T+\dfrac{2i\pi}{b\lambda_n} & \mbox{if } b<0,
\end{cases}
\qquad i=1,2,\ldots,m.
\end{equation}
By \eqref{yu-4-8-2-new}, it is straightforward to verify that $\{S_i\}_{i=1}^m \subset (0,T)$.

For each $i=1,\dots,m$, it follows from \eqref{yu-4-8-3-new} that
\begin{equation*}\label{yu-4-8-4-new}
\lambda_n b S_i=
\begin{cases}
2i\pi & \mbox{if } b>0,\\
\lambda_n b T+2i\pi & \mbox{if } b<0.
\end{cases}
\end{equation*}
In particular,
\begin{equation*}\label{yu-4-8-5-new}
\lambda_n b(S_i-S_1)\in 2\pi\mathbb{Z}
\qquad \mbox{for all } i=1,\dots,m.
\end{equation*}
Hence,
\begin{equation}\label{yu-4-8-6-new}
\cos(\lambda_n bS_i)=\cos(\lambda_n bS_1),
\;\;
\sin(\lambda_n bS_i)=\sin(\lambda_n bS_1),
\qquad i=1,\dots,m.
\end{equation}

Now we choose $(v_1,v_2)\in\mathbb{R}^2$ such that
\begin{equation}\label{yu-4-8-7-new}
v_1^2+v_2^2=1,
\;\;
v_1\cos(\lambda_n bS_1)+v_2\sin(\lambda_n bS_1)=0.
\end{equation}
For instance, one may take
$$
(v_1,v_2)=(-\sin(\lambda_n bS_1),\,\cos(\lambda_n bS_1)).
$$
Define
$$
z^*:=(v_1,v_2)^\top e_n.
$$
Then $z^*\in L^2(\Omega;\mathbb{R}^2)\setminus\{0\}$. By \eqref{yu-9-2-12}, \eqref{yu-4-8-6-new}, and \eqref{yu-4-8-7-new}, we obtain that for every $i=1,\dots,m$,
\begin{eqnarray*}
Be^{\mathcal{A}S_i}z^*
&=&
e^{-a\lambda_n S_i}
\Bigl(v_1\cos(\lambda_n bS_i)+v_2\sin(\lambda_n bS_i)\Bigr)e_n \nonumber\\
&=&
e^{-a\lambda_n S_i}
\Bigl(v_1\cos(\lambda_n bS_1)+v_2\sin(\lambda_n bS_1)\Bigr)e_n =0.
\end{eqnarray*}
This proves \eqref{yu-4-8-1} with $z=z^*$ and $\{S_i\}_{i=1}^m$ defined by \eqref{yu-4-8-3-new}. The proof is complete.
\end{proof}
\begin{remark}
    Two remarks on Proposition \ref{yu-proposition-4-8-1} are given.
\begin{enumerate}
  \item [$(i)$] In Proposition \ref{yu-proposition-4-8-1}, we show that, regardless of how many observation time points are used, the pointwise-in-time interpolation observability inequality \eqref{yu-9-2-1}, as well as the $L^2$-type multi-time-point interpolation observability inequality, may fail to hold for system \eqref{yu-8-28-1}. This stands in sharp contrast to the case of weakly coupled parabolic systems studied in \cite{Wang-Yan-Yu}. The failure is due to the strong coupling, which can induce cancellations of high-frequency oscillations in the observed component.
  \item [$(ii)$] Proposition \ref{yu-proposition-4-8-1} only demonstrates that, for any number of observation time-points, their locations can be chosen so that inequality \eqref{yu-9-2-1} fails. When the number of observation points is fixed, whether there exist alternative locations for which inequality \eqref{yu-9-2-1} is satisfied remains an open problem.
\end{enumerate}
\end{remark}

The next proposition establishes that the integral-type interpolation observability inequality remains valid for the first observation operator in \eqref{yu-9-2-2}.

\begin{proposition}\label{yu-proposition-9-2-2}
     Let $\widehat{B}$ be given in \eqref{yu-9-2-2}. Under the notation and assumptions of Theorem \ref{yu-theorem-8-18-1}, there exists a constant $M=M(a,b,\mu_1,\mu_2,D)>0$
     such that for all $\theta\in(0,1)$ and all $0<S_1<S_2\leq T$ with  $|E\cap [S_1,S_2]|\neq 0$, the following interpolation observability  inequality holds for every $z\in L^2(\Omega;\mathbb{R}^2)$:
\begin{eqnarray}\label{yu-9-3-1}
     &\;&\|e^{\mathcal{A}S_2}z\|_{L^2(\Omega;\mathbb{R}^2)}\nonumber\\
     &\leq& \frac{M e^{M\big(\frac{S_2}{1-\theta}+\frac{1}{\theta S_1}\big)}}{|E\cap [S_1,S_2]|^3}
     \Big(\frac{1}{|E\cap [S_1,S_2]|}\int_{S_1}^{S_2}\chi_E(t)\| \chi_{\mathcal{D}_t} \widehat{B} e^{\mathcal{A}t}z\|_{L^1(\Omega;\mathbb{R})}dt\Big)^{1-\theta}\|z\|_{L^2(\Omega;\mathbb{R}^2)}^{\theta}.
\end{eqnarray}
    \end{proposition}
\begin{proof}
   We arbitrarily fix $\theta\in(0,1)$,  $0<S_1<S_2\leq T$ with  $|E\cap [S_1,S_2]|\neq 0$,
   and
    $z\in L^2(\Omega;\mathbb{R}^2)$. Set
$z=\sum_{j=1}^{+\infty}(z_{1,j},z_{2,j})^\top e_j$,
    where $\{(z_{1,j},z_{2,j})^\top\}_{j=1}^{+\infty}\subset\mathbb{R}^2$.
   From \eqref{yu-9-2-12}, we have
   \begin{eqnarray}\label{yu-9-3-2}
    \widehat{B}e^{\mathcal{A}t}z=
    \sum_{j=1}^{+\infty}e^{-a\lambda_jt}
    \left[(\mu_1z_{1,j}+\mu_2z_{2,j})\cos(\lambda_jbt)+(\mu_1z_{2,j}-\mu_2z_{1,j})\sin(\lambda_jbt)\right]e_j
    \;\;\mbox{for all}\;\;t\in(0,T).
    \end{eqnarray}
    Set
$$
    \varphi:=\sum_{j=1}^{+\infty}(\mu_1z_{1,j}+\mu_2z_{2,j}, \mu_1z_{2,j}-\mu_2z_{1,j})^\top e_j.
$$
 A straightforward computation shows that
\begin{equation}\label{yu-9-3-4}
    \|\varphi\|^2_{L^2(\Omega;\mathbb{R}^2)}=\big((\mu_1)^2+(\mu_2)^2\big)\|z\|^2_{L^2(\Omega;\mathbb{R}^2)}.
\end{equation}
Two facts hold: First, combining \eqref{yu-9-2-12} and \eqref{yu-9-3-2}, we obtain
\begin{equation*}\label{yu-1-29-4}
    \widehat{B}e^{\mathcal{A}t}z=Be^{\mathcal{A}t}\varphi\;\;\mbox{for all}\;\;t\in[0,T].
\end{equation*}
    Second, by \eqref{yu-9-2-12}, we have that for every $t\in(0,T)$,
\begin{eqnarray*}\label{yu-9-3-5}
    \|e^{\mathcal{A}t}\varphi\|^2_{L^2(\Omega;\mathbb{R}^2)}&=&
    \sum_{j=1}^{+\infty}e^{-2\lambda_jat}\left|\left(
                                                 \begin{array}{cc}
                                                   \cos(\lambda_jbt) & \sin(\lambda_jbt) \\
                                                   -\sin(\lambda_jbt) & \cos(\lambda_jbt) \\
                                                 \end{array}
                                               \right)\left(
                                                        \begin{array}{c}
                                                          \mu_1z_{1,j}+\mu_2z_{2,j} \\
                                                           \mu_1z_{2,j}-\mu_2z_{1,j} \\
                                                        \end{array}
                                                      \right)\right|^2\nonumber\\
    &=&\big((\mu_1)^2+(\mu_2)^2\big)\sum_{j=1}^{+\infty}
    e^{-2\lambda_jat}\big((z_{1,j})^2+(z_{2,j})^2\big)=\big((\mu_1)^2+(\mu_2)^2\big)\|e^{\mathcal{A}t}z\|^2_{L^2(\Omega;\mathbb{R}^2)}.
\end{eqnarray*}
Combining these two facts with \eqref{yu-9-3-2} and \eqref{yu-9-3-4}, and applying Theorem \ref{yu-theorem-8-18-1} to $\varphi$, we can find  a constant $C_5=C_5(a,b,D)>0$ such that
\begin{eqnarray*}\label{yu-9-3-6}
&\;&\|e^{\mathcal{A}S_2}z\|_{L^2(\Omega;\mathbb{R}^2)}\nonumber\\
     &\leq&\frac{C_5 C_6e^{C_5\big(\frac{S_2}{1-\theta}+\frac{1}{\theta S_1}\big)}}{|E\cap [S_1,S_2]|^3}
     \Big(\frac{1}{|E\cap [S_1,S_2]|}\int_{S_1}^{S_2}\chi_E(t)\| \chi_{\mathcal{D}_t} \widehat{B} e^{\mathcal{A}t}z\|_{L^1(\Omega;\mathbb{R})}dt\Big)^{1-\theta}\|z\|_{L^2(\Omega;\mathbb{R}^2)}^{\theta},
\end{eqnarray*}
where $C_6:=\Big(1+\big((\mu_1)^2+(\mu_2)^2\big)^{-1}\Big)^{\frac{1}{2}}$, which yields   \eqref{yu-9-3-1} with $M:=C_5C_6$. This completes the proof.
\end{proof}

    The following result  establishes  the single-time-point interpolation observability inequality
    \eqref{yu-9-2-1} with $B$  replaced by  $\widetilde{B}$.

\begin{proposition}\label{yu-proposition-9-3-1}
Under the notation and assumptions of Theorem \ref{yu-theorem-8-18-1}, there exists a constant
$M=M(a,b,D)>0$ such that for each $\theta\in(0,1)$ and for a.e. $t\in E\setminus\{0\}$, the following inequality holds for any $z\in L^2(\Omega;\mathbb{R}^2)$:
\begin{equation}\label{yu-9-4-20}
\|e^{\mathcal{A} t}z\|_{L^2(\Omega;\mathbb{R}^2)}
\leq
M e^{M\big(\frac{t}{1-\theta}+\frac{1}{\theta t}\big)}
\big( \|\chi_{\mathcal{D}_t}e^{\mathcal{A} t}z\|_{L^1(\Omega;\mathbb{R}^2)}\big)^{1-\theta}
\|z\|^{\theta}_{L^2(\Omega;\mathbb{R}^2)}.
\end{equation}
\end{proposition}
\begin{proof}
    Let $z=\sum_{j=1}^{+\infty}(z_j^1,z_j^2)^\top e_j\in L^2(\Omega;\mathbb{R}^2)$ be fixed arbitrarily. For each $\lambda> \lambda_1$, we let $\varphi_1:=(\varphi_1^1,\varphi_1^2)=\sum_{j=1}^{k_\lambda}(z_j^1,z_j^2)^\top e_j$
    and $\varphi_2:=(\varphi_2^1,\varphi_2^2)=\sum_{j=k_\lambda+1}^{+\infty}(z_j^1,z_j^2)^\top e_j$.

    We first show that for a.e. $t\in E$,
    \begin{eqnarray}\label{yu-9-4-22}
    \|e^{\mathcal{A}t}\varphi_1\|_{L^2(\Omega;\mathbb{R}^2)}
    \leq 2De^{D\sqrt{\lambda}}\|\chi_{\mathcal{D}_t}e^{\mathcal{A} t}\varphi_1\|_{L^1(\Omega;\mathbb{R}^2)}.
\end{eqnarray}
Let $w(t):=e^{\mathcal{A}t}\varphi_1$. By \eqref{yu-9-2-12}, we have
$e^{\mathcal{A}t}\varphi_1=(w^1(t),w^2(t))^\top$,
where
\begin{equation}\nonumber
\begin{cases}
    w^1(t)=\sum_{j=1}^{k_\lambda}e^{-a\lambda_jt}
\big(\cos(\lambda_jbt)z_j^1+\sin(\lambda_jbt)z_j^2\big)e_j,\\
w^2(t)=\sum_{j=1}^{k_\lambda}e^{-a\lambda_jt}
\big(-\sin(\lambda_jbt)z_j^1+\cos(\lambda_jbt)z_j^2\big)e_j.
\end{cases}
\end{equation}
Since $w^1(t), w^2(t) \in \mbox{span}\{e_1,\ldots,e_{k_\lambda}\}$ for each $t\in\mathbb{R}^+$, Assumption $(c)$ in Theorem \ref{yu-theorem-8-18-1} yields that for a.e. $t\in E$,
\begin{equation*}
    \|w^i(t)\|_{L^2(\Omega;\mathbb{R})}\leq
    De^{D\sqrt{\lambda}}\|\chi_{\mathcal{D}_t}w^i(t)\|_{L^1(\Omega;\mathbb{R})},\quad\quad i=1,2.
\end{equation*}
Hence
\begin{eqnarray*}
    \|e^{\mathcal{A}t}\varphi_1\|_{L^2(\Omega;\mathbb{R}^2)}
    &=& \big(\|w^1(t)\|^2_{L^2(\Omega;\mathbb{R})}
    +\|w^2(t)\|^2_{L^2(\Omega;\mathbb{R})}\big)^\frac{1}{2}\nonumber\\
    &\leq&\|w^1(t)\|_{L^2(\Omega;\mathbb{R})}+\|w^2(t)\|_{L^2(\Omega;\mathbb{R})}\nonumber\\
    &\leq&De^{D\sqrt{\lambda}}
    \left(\|\chi_{\mathcal{D}_t}w^1(t)\|_{L^1(\Omega;\mathbb{R})}
    +\|\chi_{\mathcal{D}_t}w^2(t)\|_{L^1(\Omega;\mathbb{R})}\right)\nonumber\\
    &\leq&2De^{D\sqrt{\lambda}}\|\chi_{\mathcal{D}_t}w(t)\|_{L^1(\Omega;\mathbb{R}^2)}
    =2De^{D\sqrt{\lambda}}\|\chi_{\mathcal{D}_t}e^{\mathcal{A} t}\varphi_1\|_{L^1(\Omega;\mathbb{R}^2)}.
\end{eqnarray*}
This proves \eqref{yu-9-4-22}.

    Meanwhile, it is straightforward to check that  $\|e^{\mathcal{A}t}\varphi_2\|_{L^2(\Omega;\mathbb{R}^2)}\leq e^{-a\lambda t}\|z\|_{L^2(\Omega;\mathbb{R}^2)}$ for each $t\in\mathbb{R}^+$.
    This, along with  \eqref{yu-9-4-22} and  H\"{o}lder's inequality, implies for a.e. $t\in E$,
   \begin{eqnarray}\label{yu-9-4-6}
    \|e^{\mathcal{A} t}z\|_{L^2(\Omega;\mathbb{R}^2)}
    &\leq& \|e^{\mathcal{A} t}\varphi_1\|_{L^2(\Omega;\mathbb{R}^2)}+\|e^{\mathcal{A} t}\varphi_2\|_{L^2(\Omega;\mathbb{R}^2)}\nonumber\\
    &\leq& 2De^{D\sqrt{\lambda}}\|\chi_{\mathcal{D}_t}e^{\mathcal{A} t}z\|_{L^1(\Omega;\mathbb{R}^2)}
    +(2D|\Omega|^{\frac{1}{2}}+1)e^{D\sqrt{\lambda}-a\lambda t}\|z\|_{L^2(\Omega;\mathbb{R}^2)}.
\end{eqnarray}

   Now,  fix  an arbitrary    $\theta\in (0,1)$ and  set $\gamma:=\theta(1-\theta)^{-1}$.
   One can easily verify  that for each $t>0$,
\begin{equation*}\label{yu-9-4-7}
    D\sqrt{\lambda}-\frac{1}{2}a\lambda t\leq \frac{D^2}{a\theta t},\;\;
    D\sqrt{\lambda}
    \leq \frac{1}{2}\gamma a\lambda t+\frac{D^2}{a\theta t}\;\;\mbox{and}\;\;e^{\frac{1}{2}\gamma \lambda_1t}\leq
    e^{\frac{\lambda_1t}{1-\theta}}.
\end{equation*}
   Combining these with  \eqref{yu-9-4-6} implies that for a.e. $t\in E\setminus\{0\}$ and all $\lambda> \lambda_1$,
\begin{eqnarray}\label{yu-9-4-8}
    &\;&\|e^{\mathcal{A}t}z\|_{L^2(\Omega;\mathbb{R}^2)}\nonumber\\
    &\leq& (2D+1)(|\Omega|^{\frac{1}{2}}+1)
    e^{(a^{-1}D^2+a\lambda_1)\big(\frac{t}{1-\theta}+\frac{1}{\theta t}\big)}
    \Big(\varepsilon(\lambda)^{-\gamma}\|\chi_{\mathcal{D}_t}e^{\mathcal{A} t}z\|_{L^1(\Omega;\mathbb{R}^2)}
    +\varepsilon(\lambda)\|z\|_{L^2(\Omega;\mathbb{R}^2)}\Big),
\end{eqnarray}
    where $\varepsilon(\lambda):=e^{-\frac{1}{2}a(\lambda-\lambda_1)t}$.
 \par
    We now use Lemma \ref{yu-lemma-9-4-3} to show \eqref{yu-9-4-20}, where $$H:=L^2(\Omega;\mathbb{R}^2),
  \;\;F_1(\cdot):=\|e^{\mathcal{A} t}\cdot\|_{L^2(\Omega;\mathbb{R}^2)},$$
    $$F_2(\cdot):=\|\chi_{D_t}e^{\mathcal{A} t}\cdot\|_{L^1(\Omega;\mathbb{R}^2)}\;\;\mbox{and}\;\; F_3(\cdot):=\|\cdot\|_{L^2(\Omega;\mathbb{R}^2)}.$$
    We note that the function $\lambda\mapsto \varepsilon(\lambda)$
  for $\lambda> \lambda_1$ is surjective onto $(0,1)$, and $\|e^{\mathcal{A} t}z\|_{L^2(\Omega;\mathbb{R}^2)}\leq \|z\|_{L^2(\Omega;\mathbb{R}^2)}$ for each $z\in L^2(\Omega;\mathbb{R}^2)$ and $t\geq 0$.
     These, along with \eqref{yu-9-4-8} and the arbitrariness of $z$ in it, yields that the statement $(i)$ in Lemma \ref{yu-lemma-9-4-3} holds
     with
     $\Pi_1:=(2D+1)(|\Omega|^{\frac{1}{2}}+1)
    e^{(a^{-1}D^2+a\lambda_1)\big(\frac{t}{1-\theta}+\frac{1}{\theta t}\big)}$.
  Thus, by Lemma~\ref{yu-lemma-9-4-3}, we get that for any $z\in L^2(\Omega;\mathbb{R}^2)$,
  $$\|e^{\mathcal{A} t}z\|_{L^2(\Omega;\mathbb{R}^2)}
\leq
2(2D+1)(|\Omega|^{\frac{1}{2}}+1)
    e^{(a^{-1}D^2+a\lambda_1)\big(\frac{t}{1-\theta}+\frac{1}{\theta t}\big)}
\big( \|\chi_{\mathcal{D}_t}e^{\mathcal{A} t}z\|_{L^1(\Omega;\mathbb{R}^2)}\big)^{1-\theta}
\|z\|^{\theta}_{L^2(\Omega;\mathbb{R}^2)}.$$
  This gives
  \eqref{yu-9-4-20} with $M:=2\max\{(2D+1)(|\Omega|^{\frac{1}{2}}+1), a^{-1}D^2+a\lambda_1\}$.
We complete the proof.
\end{proof}

\section{Proof of Theorem \ref{yu-theorem-8-18-2}}\label{yu-sec-4}

\begin{proof}[Proof of Theorem \ref{yu-theorem-8-18-2}]
The proof of Theorem \ref{yu-theorem-8-18-2} relies on the integral-type interpolation observability inequality \eqref{yu-8-18-1} and a suitable application of the strategy developed in \cite{PW, Luis-Wang-Zhang} for deriving observability from measurable sets.

       By Lemma \ref{yu-lemma-9-3-1}, there exists
       $(x_0,t_0)\in \mathcal{D}$ and
        $R>0$ such that
\begin{equation}\label{yu-8-21-2}
    B_{4R}^{d+1}(x_0,t_0)\subset \Omega\times(0,T),\;\;\mbox{and}\;\;|B_R^{d+1}(x_0,t_0)\cap \mathcal{D}|\geq \frac{1}{2}|B_R^{d+1}|>0.
\end{equation}
   Let  $\mathcal{D}^R:=B_R^{d+1}(x_0,t_0)\cap \mathcal{D}$.
  Then it follows from    \eqref{yu-8-21-2} that
\begin{equation}\label{yu-8-21-3}
    B^d_{4R}(x_0)\subset \Omega, \;\;\mathcal{D}^R\subset B_R^d(x_0)\times(0,T)\;\;\mbox{and}\;\;
    |\mathcal{D}^R|>0.
\end{equation}
   For each $t\in (0,T)$, we denote
   \begin{equation}\label{yu-8-21-4}
    {\mathcal{D}^R}_t:=\{x\in\Omega:(x,t)\in \mathcal{D}^R\}\;\;\mbox{and}\;\;
    E^R:=\Big\{t\in(0,T):|{\mathcal{D}^R}_t|\geq \frac{|\mathcal{D}^R|}{2T}\Big\}.
\end{equation}
By  \eqref{yu-8-21-3} and \eqref{yu-8-21-4}, we may apply  Lemma \ref{yu-lemma-8-19-1} (with $\mathcal{D}$ replaced by $\mathcal{D}^R$) to obtain the following conclusions:
\begin{enumerate}
  \item [$(q_1)$] ${\mathcal{D}^R}_t$ is a measurable subset of $\Omega$ for a.e. $t\in(0,T)$;
  \item[$(q_2)$] for any $f\in L^1(\Omega\times(0,T))$,
      the function $t\mapsto \int_{\Omega}\chi_{{\mathcal{D}^R}_t}(x)|f(x,t)|dx$ is integrable over $(0,T)$;
  \item [$(q_3)$] $E^R$ is a measurable subset of $(0,T)$ and $|E^R|\geq \frac{|\mathcal{D}^R|}{2|B^d_R|}$;
  \item [$(q_4)$] $\chi_{E^R}\chi_{{\mathcal{D}^R}_t}\leq \chi_{\mathcal{D}^R}$ in $\Omega\times(0,T)$.
  \end{enumerate}
  Meanwhile, by \eqref{yu-8-21-3} and  \eqref{yu-8-21-4}, we may apply  Corollary \ref{yu-corollary-8-19-1}
  (with $\mathcal{D}$ replaced by $\mathcal{D}^R$) to obtain
    a constant $D(R,|\mathcal{D}^R|/T)>0$ such that for all $\lambda> \lambda_1$ and $t\in E^R$,
\begin{equation}\label{yu-8-21-5}
\sum_{i=1}^{k_\lambda}|a_i|^2 \leq D(R,|\mathcal{D}^R|/T) e^{D(R,|\mathcal{D}^R|/T) \sqrt{\lambda}} \Big(\Big\|\chi_{{\mathcal{D}^R}_t}\sum_{i=1}^{k_\lambda} a_i e_i\Big\|_{L^1(\Omega;\mathbb{R})}\Big)^2\;\;\mbox{for all}\;\;
\{a_i\}_{i=1}^{k_\lambda}\subset\mathbb{R}.
\end{equation}

The remainder of the proof is divided into two steps:

\vskip 5pt
  \noindent  \emph{Step 1. We  show that there exists a constant $M=M(a,b,T,\mathcal{D})>0$ such that
\begin{equation}\label{yu-8-21-7}
    \|e^{\mathcal{A}T}z\|_{L^2(\Omega;\mathbb{R}^2)}
    \leq M\int_0^T\chi_{E^R}(t)\|\chi_{{\mathcal{D}^R}_t} Be^{\mathcal{A}t}z\|_{L^1(\Omega;\mathbb{R})}dt
    \;\;\mbox{for all}\;\;z\in L^2(\Omega;\mathbb{R}^2).
\end{equation}}
   Here, we note that, by $(q_1)$ and $(q_2)$, the integral on the right-hand side of inequality \eqref{yu-8-21-7} is well defined. Fix an arbitrary \(\beta>0\), and set
  \begin{equation}\label{wang-4.6-2-6}
    \mu:=\sqrt{\frac{\beta+2}{\beta+1}}.
    \end{equation}
    Let  $\ell$ be a density point of $E^R$. By Lemma \ref{a2} (with $E$ replaced by $E^R$),
    there exists a strictly monotonically decreasing
    sequence $\{\ell_m\}_{m=1}^{+\infty}$, given by
        \eqref{b1} with $\mu$ defined in \eqref{wang-4.6-2-6}, satisfying
        \begin{equation}\label{yu-8-21-8}
    \ell_m-\ell_{m+1}\leq3|E^R\cap(\ell_{m+1},\ell_m)|.
\end{equation}
       For each $m\in\mathbb{N}^+$, we define the following two sets:
\begin{equation}\label{yu-8-18-10}
    \widehat{E}_m^R:=E^R\cap (\ell_{m+1},\ell_{m})-\ell_{m+2}\subset(0,\ell_m-\ell_{m+2}),
\end{equation}
where the inclusion follows from  the strict monotonicity of the decreasing sequence
  $\{\ell_m\}_{m=1}^{+\infty}$, and
\begin{equation}\label{yu-8-18-10-bb}
\widehat{\mathcal{D}}^R{}_{t,m}:={\mathcal{D}^R}_{t+\ell_{m+2}},\;\;t\in(0,\ell_m-\ell_{m+2}).
\end{equation}
   Here, the set on the right-hand side of \eqref{yu-8-18-10-bb} is defined in \eqref{yu-8-21-4}.

   We state three facts below:

    First, it  is clear that
\begin{equation}\label{yu-8-18-11}
    \ell < \ell_m < T\;\;\mbox{and}\;\;|\widehat{E}^R_m|=|E^R\cap (\ell_{m+1},\ell_{m})|\;\;\mbox{for all}\;\;m\in\mathbb{N}^+.
\end{equation}

Second, there exists a constant $C_7=C_7(a,b,T,\mathcal{D})>0$ such that  for all $m\in \mathbb{N}^+$, $\theta\in (0,1)$ and $z\in L^2(\Omega;\mathbb{R}^2)$,
\begin{eqnarray}\label{yu-8-18-12}
   \|e^{\mathcal{A}\ell_m}z\|_{L^2(\Omega;\mathbb{R}^2)}
   \leq
   N_4
       \Big(\int_{\ell_{m+1}-\ell_{m+2}}^{\ell_{m}-\ell_{m+2}}
    \chi_{\widehat{E}^R_m}(t)\|
    \chi_{\widehat{\mathcal{D}}^R{}_{t,m}} B e^{\mathcal{A}(t+\ell_{m+2})}z\|_{L^1(\Omega;\mathbb{R})}dt\Big)^{1-\theta}
       \|e^{\mathcal{A}\ell_{m+2}}z\|_{L^2(\Omega;\mathbb{R}^2)}^{\theta}
\end{eqnarray}
where
\begin{eqnarray*}
N_4:=\frac{C_7 e^{C_7\big(\frac{\ell_{m}-\ell_{m+2}}{1-\theta}
+\frac{(\ell_{m+1}-\ell_{m+2})^{-1}}{\theta}\big)}}{|\widehat{E}^R_m|^{4-\theta}}.
\end{eqnarray*}
   Indeed, we fix arbitrarily $m\in \mathbb{N}^+$, $\theta\in (0,1)$ and $z\in L^2(\Omega;\mathbb{R}^2)$.
   From $(q_1)$, \eqref{yu-8-21-5}, and \eqref{yu-8-18-10-bb}, it follows  that for a.e.  $t\in (0,\ell_m-\ell_{m+2})\bigcap \widehat{E}^R_{m}$, $\widehat{\mathcal{D}}^R{}_{t,m}$ is measurable, and for all $\lambda>\lambda_1$,
\begin{equation*}\label{yu-8-21-10}
    \sum_{i=1}^{k_\lambda}|a_i|^2 \leq D(R,|\mathcal{D}^R|/T) e^{D(R,|\mathcal{D}^R|/T) \sqrt{\lambda}} \Big(\Big\|\chi_{\widehat{\mathcal{D}}^R{}_{t,m}}\sum_{i=1}^{k_\lambda} a_i e_i\Big\|_{L^1(\Omega;\mathbb{R})}\Big)^2\;\;\mbox{for all}\;\;
\{a_i\}_{i=1}^{k_\lambda}\subset\mathbb{R}.
\end{equation*}
Combining these with  $(q_2)$, $(q_3)$ and \eqref{yu-8-18-10}, we may apply Theorem  \ref{yu-theorem-8-18-1}
 with the following identifications
\begin{equation*}
T=\ell_{m}-\ell_{m+2};\;E=\widehat{E}_m^R;\;\mathcal{D}_t=
\widehat{\mathcal{D}}^R{}_{t,m};\;S_1=\ell_{m+1}-\ell_{m+2};\;
S_2=\ell_{m}-\ell_{m+2};\; D=D(R,|\mathcal{D}^R|/T)
\end{equation*}
to derive a constant
    $C_8=C_8(a,b,T,\mathcal{D})>0$ such that
\begin{eqnarray*}\label{yu-8-21-6}
\|e^{\mathcal{A}(\ell_m-\ell_{m+2})}z\|_{L^2(\Omega;\mathbb{R}^2)}
     \leq \frac{N_5}{|\widehat{E}_m^R|^3} \Big(\frac{1}{|\widehat{E}_m^R|}
     \int_{\ell_{m+1}-\ell_{m+2}}^{\ell_{m}-\ell_{m+2}}\chi_{\widehat{E}_m^R}(t)\| \chi_{\widehat{\mathcal{D}}^R{}_{t,m}} B e^{\mathcal{A}t}z\|_{L^1(\Omega;\mathbb{R})}dt\Big)^{1-\theta}
     \|z\|_{L^2(\Omega;\mathbb{R}^2)}^{\theta},
\end{eqnarray*}
where
\begin{eqnarray*}
N_5:=C_8 e^{C_8\big(\frac{\ell_m-\ell_{m+2}}{(1-\theta)}+\frac{(\ell_{m+1}-\ell_{m+2})^{-1}}{\theta} \big)}.
\end{eqnarray*}
Replacing $z$ with $e^{\mathcal{A}\ell_{m+2}}z$ in the above inequality yields \eqref{yu-8-18-12}
with $C_7:=C_8$.

Third, it
 follows  directly
  from  \eqref{b1} and \eqref{yu-8-21-8} that
\begin{equation}\label{yu-8-18-14}
\begin{cases}
(\ell_{m+1}-\ell_{m+2})^{-1}=(\mu-1)^{-1}(\ell_1-\ell)^{-1}\mu^{m+1}\leq(\mu-1)^{-1}(\ell_1-\ell)^{-1}\mu^{m+2}\\
(|E^R\cap (\ell_{m+1},\ell_{m})|)^{-1}\leq 3(\ell_1-\ell)^{-1}(\mu-1)^{-1}\mu^m
\end{cases}\;\;\mbox{for all}\;\;m\in\mathbb{N}^+.
\end{equation}

   Combining the above facts \eqref{yu-8-18-11}, \eqref{yu-8-18-12} (with $\theta=1-(1+\beta)^{-1}$ and $\beta$ as defined in \eqref{wang-4.6-2-6}), and \eqref{yu-8-18-14}, together with
    \eqref{yu-8-18-10}, \eqref{yu-8-18-10-bb}, and  Young's inequality, we obtain a constant
       $C_9=C_9(\beta,\ell,\ell_1)>0$ such that for all $\varepsilon>0$, $m\in\mathbb{N}^+$ and $z\in L^2(\Omega;\mathbb{R}^2)$,
\begin{eqnarray*}\label{yu-8-18-13}
    &\;&\|e^{\mathcal{A}\ell_m}z\|_{L^2(\Omega;\mathbb{R}^2)}\nonumber\\
    &\leq&\frac{\big(C_7e^{C_7\big((1+\beta)T+(\beta+1)\beta^{-1}(\mu-1)^{-1}(\ell_1-\ell)^{-1}\mu^{m+2}
    \big)}\big)^{\beta+1}}{\varepsilon^\beta|E^R\cap (\ell_{m+1},\ell_{m})|^{3\beta+4}}\nonumber\\
    &\;&\times\int_{\ell_{m+1}}^{\ell_{m}}\chi_{E^R\cap (\ell_{m+1},\ell_{m})}(t)\|\chi_{{\mathcal{D}^R}_t} B e^{\mathcal{A}t}z\|_{L^1(\Omega;\mathbb{R})}dt
    +\varepsilon\|e^{\mathcal{A}\ell_{m+2}}z\|_{L^2(\Omega;\mathbb{R}^2)}\nonumber\\
    &\leq&\frac{(C_7)^{\beta+1}3^{(3\beta+4)}\mu^{(3\beta+4)m}e^{C_7
    \big((\beta+1)^2T+(\beta+1)^2\beta^{-1}(\mu-1)^{-1}(\ell_1-\ell)^{-1}
    \mu^{m+2}\big)}}
    {\varepsilon^\beta[(\ell_1-\ell)(\mu-1)]^{3\beta+4}}\nonumber\\
    &\;&\times\int_{\ell_{m+1}}^{\ell_{m}} \chi_{E^R\cap (\ell_{m+1},\ell_{m})}(t)\|\chi_{{\mathcal{D}^R}_t} B e^{\mathcal{A}t}z\|_{L^1(\Omega;\mathbb{R})}dt
    +\varepsilon\|e^{\mathcal{A}\ell_{m+2}}z\|_{L^2(\Omega;\mathbb{R}^2)}\nonumber\\
    &\leq&\frac{(C_7)^{\beta+1}e^{C_7(\beta+1)^2T+C_9
    \mu^{m+2}}}
    {\varepsilon^\beta[(\ell_1-\ell)(\mu-1)]^{3\beta+4}}
    \int_{\ell_{m+1}}^{\ell_{m}} \chi_{E^R\cap (\ell_{m+1},\ell_{m})}(t)\|\chi_{{\mathcal{D}^R}_t} B e^{\mathcal{A}t}z\|_{L^1(\Omega;\mathbb{R})}dt
    +\varepsilon\|e^{\mathcal{A}\ell_{m+2}}z\|_{L^2(\Omega;\mathbb{R}^2)}.
\end{eqnarray*}
    This implies that for all $\varepsilon>0$, $m\in\mathbb{N}^+$ and $z\in L^2(\Omega;\mathbb{R}^2)$,
\begin{eqnarray}\label{yu-8-19-1}
    &\;&\varepsilon^\beta e^{-C_9\mu^{m+2}}\|e^{\mathcal{A}\ell_m}
    z\|_{L^2(\Omega;\mathbb{R}^2)}-\varepsilon^{\beta+1} e^{-C_9\mu^{m+2}}\|e^{\mathcal{A}\ell_{m+2}}
    z\|_{L^2(\Omega;\mathbb{R}^2)}\nonumber\\
    &\leq&\frac{(C_7)^{\beta+1}e^{C_7(\beta+1)^2T}}{[(\ell_1-\ell)(\mu-1)]^{3\beta+4}}
    \int_{\ell_{m+1}}^{\ell_{m}} \chi_{E^R\cap (\ell_{m+1},\ell_{m})}(t)\|\chi_{{\mathcal{D}^R}_t} B e^{\mathcal{A}t}z\|_{L^1(\Omega;\mathbb{R})}dt.
\end{eqnarray}
   Meanwhile, for  $m\in \mathbb{N}^+$, we  set
   \begin{equation}\label{yu-8-19-2}
     \varepsilon_m:=e^{-C_9\mu^{m+2}}.
\end{equation}
   A straightforward calculation yields
\begin{equation}\label{yu-8-19-3}
(\varepsilon_m)^\beta e^{-C_9\mu^{m+2}}
=e^{-C_9(\beta+2)\mu^{m}} \quad \text{and} \quad
(\varepsilon_m)^{\beta+1} e^{-C_9\mu^{m+2}}=e^{-C_9(\beta+2)\mu^{m+2}};
\end{equation}
\begin{equation}\label{yu-8-19-4}
e^{-C_9(\beta+2)\mu^{m}}\to0 \quad \text{as } m\to+\infty.
\end{equation}

We now fix arbitrarily   $m\in\mathbb{N}^+$ and $z\in L^2(\Omega;\mathbb{R}^2)$. By   \eqref{yu-8-19-1} \big(with $\varepsilon$
replaced by $\varepsilon_m$ which is defined in \eqref{yu-8-19-2}\big)
   and
     \eqref{yu-8-19-3}, we obtain
     \begin{eqnarray*}\label{yu-8-19-5}
    &\;&e^{-C_9(\beta+2)\mu^{m}}\|e^{\mathcal{A}\ell_m}
    z\|_{L^2(\Omega;\mathbb{R}^2)}-e^{-C_9(\beta+2)\mu^{m+2}}\|e^{\mathcal{A}\ell_{m+2}}
    z\|_{L^2(\Omega;\mathbb{R}^2)}\nonumber\\
    &\leq&\frac{(C_7)^{\beta+1}e^{C_7(\beta+1)^2T}}{[(\ell_1-\ell)(\mu-1)]^{3\beta+4}}
    \int_{\ell_{m+1}}^{\ell_{m}} \chi_{E^R\cap (\ell_{m+1},\ell_{m})}(t)\|\chi_{{\mathcal{D}^R}_t} B e^{\mathcal{A}t}z\|_{L^1(\Omega;\mathbb{R})}dt.
\end{eqnarray*}
    Since $\|e^{\mathcal{A}t}\|_{\mathcal{L}(L^2(\Omega;\mathbb{R}^2))}
    \leq 1$ for all $t\in\mathbb{R}^+$, the above inequality together with \eqref{yu-8-19-4}
    implies that
   \begin{eqnarray*}\label{yu-8-19-6}
   e^{-C_9(\beta+2)\mu^{2}}\|e^{\mathcal{A}\ell_2}z\|_{L^2(\Omega;\mathbb{R}^2)}
    &\leq&\frac{(C_7)^{\beta+1}e^{C_7(\beta+1)^2T}}{[(\ell_1-\ell)(\mu-1)]^{3\beta+4}}
    \sum_{j=1}^{+\infty}
    \int_{\ell_{2j+1}}^{\ell_{2j}}\chi_{E^R\cap (\ell_{2j+1},\ell_{2j})}(t)\|\chi_{{\mathcal{D}^R}_t} B e^{\mathcal{A}t}z\|_{L^1(\Omega;\mathbb{R})}dt\nonumber\\
    &\leq&\frac{(C_7)^{\beta+1}e^{C_7(\beta+1)^2T}}{[(\ell_1-\ell)(\mu-1)]^{3\beta+4}}
    \int_{0}^{T}\chi_{E^R}(t)\|\chi_{{\mathcal{D}^R}_t} B e^{\mathcal{A}t}z\|_{L^1(\Omega;\mathbb{R})}dt.
\end{eqnarray*}
 Combining this with  the inequality $\|e^{\mathcal{A}T}z\|_{L^2(\Omega;\mathbb{R}^2)}\leq
    \|e^{\mathcal{A}\ell_2}z\|_{L^2(\Omega;\mathbb{R}^2)}$, yields \eqref{yu-8-21-7} with
    \begin{eqnarray*}
    M=M(a,b,T,\mathcal{D}):=\frac{(C_7)^{\beta+1}e^{\big(C_7(\beta+1)^2T+C_9(\beta+2)\mu^2\big)}}
    {[(\ell_1-\ell)(\mu-1)]^{3\beta+4}}.
 \end{eqnarray*}
\vskip 5pt
   \noindent \emph{Step 2. We prove  \eqref{yu-8-28-2}.}

    Since  $\mathcal{D}^R\subset \mathcal{D}$,  it follows from the result $(q_4)$ (cf. \emph{Step 1}) and \eqref{yu-8-21-7} that
\begin{equation*}\label{yu-8-22-1}
     \|y(T;y_0)\|_{L^2(\Omega;\mathbb{R}^2)}\leq M\int_{\Omega\times(0,T)}\chi_{\mathcal{D}}(x,t)| By(x,t;y_0)|dxdt\;\;\mbox{for all}\;\;y_0\in L^2(\Omega;\mathbb{R}^2).
\end{equation*}
   Combining this with \eqref{yu-9-2-11} immediately yields \eqref{yu-8-28-2}.

This completes the proof.
\end{proof}

\par
By Proposition \ref{yu-proposition-9-2-2} and the same argument as in the proof of Theorem \ref{yu-theorem-8-18-2}, we obtain the following result:

\begin{proposition}\label{yu-proposition-9-8-1}
Let $\widehat{B}$ be as defined in \eqref{yu-9-2-2}, and let $\mathcal{D}\subset\Omega\times (0,T)$ be a set of positive measure. There exists a constant $N=N(a,b,T,\mu_1,\mu_2,\mathcal{D})>0$ such that
\begin{equation}\label{yu-9-8-5}
\|y(T;y_0)\|_{L^2(\Omega;\mathbb{R}^2)}
\leq N\int_{\Omega\times(0,T)}\chi_{\mathcal{D}}(x,t)|\widehat{B}y(x,t;y_0)|dxdt
\quad\text{for all } y_0\in L^2(\Omega;\mathbb{R}^2).
\end{equation}
\end{proposition}

The following is a direct consequence of Proposition \ref{yu-proposition-9-8-1}:

\begin{corollary}\label{yu-corollary-9-8-1}
Let $\mathcal{D}\subset\Omega\times (0,T)$ be a set of positive measure, and let $K$ be a non-zero $2\times 2$ real matrix. There exists a constant $N=N(a,b,T,K,\mathcal{D})>0$ such that \eqref{yu-9-8-5} holds with $\widehat{B}$ replaced by $K$.
\end{corollary}

\section{Applications}\label{yu-sec-5}
The observability inequality established in Theorem  \ref{yu-theorem-8-18-2} implies $L^\infty$-null controllability, which in turn yields the bang-bang property of time-optimal control problem via a standard variational argument. In this section, we consider these problems.

    Let  $\mathcal{D}\subset \Omega\times(0,T)$ and $\omega\subset\Omega$ be two sets of positive measure in $\mathbb{R}^{d+1}$ and $\mathbb{R}^{d}$, respectively.
     Consider the following controlled coupled parabolic systems with  initial state $v_0:=(v_{0,1},v_{0,2})^\top\in L^2(\Omega;\mathbb{R}^2)$:
\begin{equation}\label{yu-9-8-10}
\begin{cases}
    \partial_tv_1=a\Delta v_1+b\Delta v_2+\chi_{\mathcal{D}}u&\mbox{in}\;\;\Omega\times(0,T),\\
    \partial_tv_2=-b\Delta v_1+a\Delta v_2 &\mbox{in}\;\;\Omega\times(0,T),\\
    (v_1,v_2)^\top=0&\mbox{on}\;\;\partial\Omega\times(0,T),\\
(v_1(0),v_2(0))^\top=v_0 &\mbox{in}\;\;\Omega.
\end{cases}
\end{equation}
    where $u\in L^\infty(\Omega\times(0,T);\mathbb{R})$ is a control, and
    \begin{equation}\label{wang5.2-2-7}
\begin{cases}
    \partial_tv_1=a\Delta v_1+b\Delta v_2+\chi_{\omega}u&\mbox{in}\;\;\Omega\times\mathbb{R}^+,\\
    \partial_tv_2=-b\Delta v_1+a\Delta v_2 &\mbox{in}\;\;\Omega\times\mathbb{R}^+,\\
    (v_1,v_2)^\top=0&\mbox{on}\;\;\partial\Omega\times\mathbb{R}^+,\\
(v_1(0),v_2(0))^\top=v_0 &\mbox{in}\;\;\Omega,
\end{cases}
\end{equation}
    where $u$  is a control taken from the admissible control set
        \begin{equation*}\label{yu-9-10-1}
    \mathcal{U}_{ad}:=\{u\in L^\infty(\Omega\times\mathbb{R}^+;\mathbb{R}):\nu_1\leq u(x,t)\leq \nu_2\;\;\mbox{for a.e.}\;\;(x,t)\in\Omega\times\mathbb{R}^+\}\;\;\mbox{with}\;\;\nu_1<\nu_2.
\end{equation*}

    We denote by $v(\cdot;u,v_0)=(v_1(\cdot;u,v_0),v_2(\cdot;u,v_0))^\top$ the solution
    to
    control systems \eqref{yu-9-8-10} and  \eqref{wang5.2-2-7} corresponding to the initial state
    $v_0$ and  control $u$.
     We note that the control operators in   equations \eqref{yu-9-8-10} and  \eqref{wang5.2-2-7}
     are  $\chi_\mathcal{D}B^\top$ and $\chi_\omega B^\top$, respectively, where $B$ is defined in \eqref{yu-9-2-11}. The dual equation of \eqref{yu-9-8-10} is given by \eqref{yu-8-28-1}. Moreover, the solution
     $ v(t;u,v_0)$ to  \eqref{wang5.2-2-7} admits the representation
    \begin{equation}\label{yu-9-8-12}
   v(t;u,v_0)=e^{\mathcal{A}^*t}v_0+\int_0^te^{\mathcal{A}^*(t-s)}\chi_{\omega}B^\top u(s)ds
    \;\;\mbox{for all}\;\;t\in\mathbb{R}^+,
\end{equation}
    where $\mathcal{A}$ is defined in \eqref{yu-9-2-10}.

\subsection{$L^\infty$-null controllability}
We consider the null controllability of system \eqref{yu-9-8-10}.
The following results are direct consequences of Theorem \ref{yu-theorem-8-18-2}.

\begin{theorem}\label{yu-theorem-9-8-1}
    The following statements hold:
\begin{enumerate}
  \item [$(i)$] It holds that $L>0$, where
\begin{equation}\label{yu-2-28-1}
    L:=\inf_{y_0\in L^2(\Omega;\mathbb{R}^2)\setminus\{0\}}
    \frac{\int_{\Omega\times(0,T)}\chi_{\mathcal{D}}(x,t)|y_1(x,t;y_0)|dxdt}
    {\|y(T;y_0)\|_{L^2(\Omega;\mathbb{R}^2)}}.
\end{equation}
  \item [$(ii)$] For each $v_0\in L^2(\Omega;\mathbb{R}^2)$, there exists a control $u\in L^\infty((0,T)\times\Omega;\mathbb{R})$, satisfying
     \begin{equation*}\label{yu-9-11-4}
    \|u\|_{ L^\infty((0,T)\times\Omega;\mathbb{R})}\leq L^{-1}\|v_0\|_{L^2(\Omega;\mathbb{R}^2)},
\end{equation*}
          such that the solution
    $v(\cdot;u,v_0)$ to  \eqref{yu-9-8-10} satisfies
              $v(T;u,v_0)=0$.
\end{enumerate}
    \end{theorem}
\begin{proof}
    We first prove $(i)$. Let $N>0$ be given such that \eqref{yu-8-28-2} holds, whose existence is guaranteed by Theorem
    \ref{yu-theorem-8-18-2}. Then, we have
    $$
    \frac{\int_{\Omega\times(0,T)}\chi_{\mathcal{D}}(x,t)|y_1(x,t;y_0)|dxdt}
    {\|y(T;y_0)\|_{L^2(\Omega;\mathbb{R}^2)}}\geq N^{-1}
    \;\;\mbox{for any}\;\;y_0\in L^2(\Omega;\mathbb{R}^2)\setminus\{0\}.$$
    By \eqref{yu-2-28-1}, it follows that $L\geq N^{-1}>0$, which proves $(i)$.
\par

    We now prove $(ii)$. From $(i)$, we have
$$
    \|y(T;y_0)\|_{L^2(\Omega;\mathbb{R}^2)}\leq L^{-1}\int_{\Omega\times(0,T)}\chi_{\mathcal{D}}(x,t)|y_1(x,t;y_0)|dxdt
\quad\text{for all } y_0\in L^2(\Omega;\mathbb{R}^2).
$$
    By the standard duality argument (see, for instance, \cite[Lemma 5.1]{Wang-Wang-Zhang}), $(ii)$ is well known. Thus, we omit its proof.
\end{proof}

\subsection{Time-optimal control problem}
 Let
 $\mathcal{S}$ be a target set, which is assumed to be  closed in $L^2(\Omega;\mathbb{R}^2)$.
We fix an arbitrary  $v_0\in L^2(\Omega;\mathbb{R}^2)\setminus\mathcal{S}$ and consider the following time-optimal control problem:
$$
     \mbox{\textbf{(TP)}}\;\; T^*:=\inf_{u\in \mathcal{U}_{ad}}\{T>0:v(T;u,v_0)\in \mathcal{S}\},
$$
   where $v(\cdot;u,v_0)$ stands for the solution to  \eqref{wang5.2-2-7}. In what follows, we investigate the bang-bang property of optimal controls for problem $\textbf{(TP)}$.

\begin{theorem}\label{yu-theorem-9-10-1}
Let $T^*$ and $u^*$ be the optimal time and an optimal control for \textbf{(TP)}, respectively. Then $u^*$
possesses the bang-bang property, i.e.,
     $u^*(x,t)\in\{\nu_1,\nu_2\}$ for a.e. $(x,t)\in\omega\times(0,T^*)$.
     \end{theorem}
\begin{proof}
    Since $v_0\in L^2(\Omega;\mathbb{R}^2)\setminus\mathcal{S}$, it is clear that $T^*>0$.

We now argue by contradiction and suppose that the bang-bang property does not hold.
    Then, there exist $\varepsilon>0$ and a set $\mathcal{D}\subset \omega\times(0,T^*)$ of positive measure such that
\begin{equation}\label{yu-9-10-2}
    \nu_1+\varepsilon<u^*(x,t)<\nu_2-\varepsilon\;\;\mbox{for all}\;\;(x,t)\in \mathcal{D}.
\end{equation}
  We aim to find $\delta^\sharp\in(0,T^*)$ and $w^\sharp\in \mathcal{U}_{ad}$ such that
\begin{equation}\label{yu-9-11-1}
    v(T^*-\delta^\sharp;w^\sharp,v_0)=v(T^*;u^*,v_0)(\in\mathcal{S}).
\end{equation}
  Once this is achieved, we reach a contradiction to the optimality of $T^*$,
 which completes the proof.

 To prove \eqref{yu-9-11-1}, we first establish some preliminary results.
  By Lemma \ref{yu-lemma-9-3-1}, there exist $R>0$ and $(x_0,t_0)\in \mathcal{D}$ such that
  \begin{equation}\label{wang-5-7-2-8}
   B_{4R}^{d+1}(x_0,t_0)\subset \Omega\times (0,T^*);\;\;|B_R^{d+1}(x_0,t_0)\cap \mathcal{D}|>0.
\end{equation}
We define the following set and its time slices:
\begin{equation*}\label{yu-9-10-3}
    \mathcal{D}^R:=B_R^{d+1}(x_0,t_0)\cap \mathcal{D};\;\;\;  {\mathcal{D}^R}_t:=\{x\in\Omega:(x,t)\in\mathcal{D}^R\}\;\;\mbox{with}\;\;t\in (0,T^*),
   \end{equation*}
and  further define the set
    \begin{equation*}\label{yu-9-10-4}
    E^R:=\left\{t\in(0,T^*):|{\mathcal{D}^R}_t|\geq \frac{|\mathcal{D}^R|}{2T^*}\right\}.
\end{equation*}
By the above definitions  and \eqref{wang-5-7-2-8}, we  apply Lemma \ref{yu-lemma-8-19-1} to obtain
the following properties:
\begin{enumerate}
      \item [$(c_1)$] ${\mathcal{D}^R}_t$ is measurable for a.e. $t\in (0,T^*)$;
      \item[$(c_2)$]
      the function $t\mapsto \int_{\Omega}\chi_{{\mathcal{D}^R}_t}(x)dx$ is integrable over $(0,T^*)$;
      \item [$(c_3)$] $E^R$ is measurable and $|E^R|\geq \frac{|\mathcal{D}^R|}{2|B_R^d|}$.
\end{enumerate}
 By  $(c_3)$, there exists $0<\delta_0<\min\{1,T^*\}$ such that
\begin{equation}\label{yu-9-10-5}
    |E^R_\delta|\geq  \frac{|\mathcal{D}^R|}{4|B_R^d|}\;\;\mbox{for all}\;\;\delta\in(0,\delta_0],\;\;
    \mbox{where}\;\;E^R_\delta:=(E^R-\delta)\cap\mathbb{R}^+.
\end{equation}
    Moreover, for all $0\leq\delta_1<\delta_2\leq \delta_0$ and $\delta\in(0,\delta_0]$,
    \begin{equation}\label{yu-9-11-b-1}
    \mathcal{Q}^R_{\delta_2}\subset\mathcal{Q}^R_{\delta_1}\;\;\mbox{and} \;\; \mathcal{Q}^R_{\delta}\subset\mathcal{D}_\delta:=\{(x,t-\delta)
    \in\Omega\times(0,T^*-\delta):(x,t)\in\mathcal{D}\},
\end{equation}
where
 \begin{equation*}
 \mathcal{Q}^R_\delta:=\{(x,t):t\in E^R_\delta,\;x\in{{\mathcal{D}}^R}_{t+\delta}\},\;\delta\in(0,\delta_0].
  \end{equation*}
   By the definition of $E^R_{\delta}$, $(c_1)$ and $(c_2)$, we have that, for each $\delta\in (0,\delta_0]$, ${\mathcal{D}^R}_{t+\delta}$
 is measurable for a.e. $t\in (0,T^*-\delta)$, and the function  $t\mapsto \int_\Omega\chi_{{\mathcal{D}^R}_{t+\delta}}(x)dx$ is integrable over $(0,T^*-\delta)$. This implies that $\mathcal{Q}^R_\delta$ is measurable for all $\delta\in (0,\delta_0]$.

  We  claim that
\begin{equation}\label{yu-9-11-b-2}
    |\mathcal{Q}^R_\delta|>0\;\;\mbox{for all}\;\;\delta\in(0,\delta_0].
\end{equation}
Indeed, combining  $(c_1)$, $(c_2)$ with \eqref{yu-9-10-5} and the definition of
 $E^R$,
 we obtain
 \begin{equation*}
 |\mathcal{Q}^R_\delta|=\int_0^{T^*-\delta}\int_{\Omega}\chi_{E^R_\delta}(t)
    \chi_{{\mathcal{D}^R}_{t+\delta}}(x)dxdt \geq \frac{|\mathcal{D}^R|^2}{8T^*|B_R^d|}
    \;\;\mbox{for all}\;\;\delta\in(0,\delta_0],
  \end{equation*}
 which yields \eqref{yu-9-11-b-2}.

 We now proceed to construct the desired quantity
  $\delta^\sharp$ and control $w^\sharp$. To this end, we first observe  from  \eqref{yu-9-8-12} that for all
 $\delta\in(0,\delta_0]$ and  $w\in \mathcal{U}_{ad}$,
     \begin{eqnarray}\label{yu-9-11-2}
    &\;&v(T^*;u^*,v_0)-v(T^*-\delta;w,v_0)\nonumber\\
    &=&e^{\mathcal{A}^*(T^*-\delta)}v_\delta
   +\int_0^{T^*-\delta}e^{\mathcal{A}^*(T^*-\delta-s)}\chi_{\omega}B^\top
    [u^*(s+\delta)-w(s)]ds,
\end{eqnarray}
    where
\begin{equation}\label{yu-9-11-3}
    v_\delta:=(e^{\mathcal{A}^*\delta}-I)v_0+\int_0^\delta e^{\mathcal{A}^*(\delta-s)}\chi_{\omega}B^\top u^*(s)ds.
\end{equation}
 For each $\delta\in (0,\delta_0]$, let  $L_\delta$ be defined by \eqref{yu-2-28-1} where $T$ and $\mathcal{D}$ are replaced by $T^*-\delta$
   and $Q_\delta^R$, respectively.
   A direct verification shows that
   \begin{equation}\label{yu-9-11-5}
    0<L_{\delta_2}\leq L_{\delta_1},\;\;\mbox{when}\;\;0< \delta_1<\delta_2\leq \delta_0.
\end{equation}
  For each $\delta\in (0,\delta_0]$, using \eqref{yu-9-11-b-2} and \eqref{yu-9-11-5}, we  apply Theorem \ref{yu-theorem-9-8-1} (with $\mathcal{D}$ and $T$ replaced by $Q^R_\delta$ and $T^*-\delta$,
  respectively) to  obtain
  a control
        $f_\delta\in L^\infty(\Omega\times(0,T^*-\delta);\mathbb{R})$ such that
\begin{equation}\label{yu-9-11-9}
    \|f_\delta\|_{L^\infty(\Omega\times(0,T^*-\delta);\mathbb{R})}\leq
    (L_{\delta_0})^{-1}\|v_\delta\|_{L^2(\Omega;\mathbb{R}^2)}
\end{equation}
    and
\begin{eqnarray}\label{yu-9-11-10}
    e^{\mathcal{A}^*(T^*-\delta)}v_\delta&=& \int_0^{T^*-\delta}e^{\mathcal{A}^*(T^*-\delta-s)}
    \chi_{\mathcal{Q}_\delta^R}(\cdot,s)B^\top f_\delta(s)ds\nonumber\\
    &=&\int_0^{T^*-\delta}e^{\mathcal{A}^*(T^*-\delta-s)}
    \chi_{\omega}B^\top \left(\chi_{\mathcal{Q}_\delta^R}(\cdot,s)f_\delta(s)\right)ds.
\end{eqnarray}
    (Here, we used the inclusion  $\mathcal{Q}_\delta^R\subset \omega\times (0,T^*)$, which follows from \eqref{yu-9-11-b-1} and the assumption $\mathcal{D}\subset \omega\times (0,T^*)$.)

    By \eqref{yu-9-11-3}, there exists  $\delta^\sharp\in(0,\delta_0]$ such that
\begin{equation}\label{yu-9-11-11}
    \|v_{\delta^{\sharp}}\|_{L^2(\Omega;\mathbb{R}^2)}\leq \frac{1}{2}\varepsilon L_{\delta_0}.
\end{equation}
  We then  define
\begin{equation}\label{yu-9-11-12}
    w^\sharp (x,t):=
\begin{cases}
    u^*(x,t+\delta^\sharp)+\chi_{\mathcal{Q}_{\delta^\sharp}^R}(x,t)f_{\delta^{\sharp}}(x,t)
    &\mbox{if}\;\;(x,t)\in\Omega\times(0,T^*-\delta^\sharp],\\
    \frac{\nu_1+\nu_2}{2}&\mbox{if}\;\;(x,t)\in\Omega\times(T^*-\delta^\sharp,+\infty).
\end{cases}
\end{equation}
    Since $u^*$, $\chi_{\mathcal{Q}_{\delta^\sharp}^R}$ and $f_{\delta^{\sharp}}$ are measurable, it is clear that
    $w^\sharp$ is measurable over $\Omega\times\mathbb{R}^+$.

 We now  claim that
\begin{eqnarray}\label{yu-9-11-13}
    \nu_1\leq w^\sharp(x,t)\leq\nu_2 \;\;\mbox{for a.e.}\;\;(x,t)\in \Omega\times\mathbb{R}^+.
\end{eqnarray}
    To prove \eqref{yu-9-11-13}, it suffices to show
   \begin{eqnarray}\label{wang-5-19-2-9}
    w^\sharp (x,t)\in \left[\nu_1+\frac{\varepsilon}{2}, \nu_2-\frac{\varepsilon}{2}\right]\;\;\mbox{for a.e.}\;\;(x,t)\in \mathcal{Q}_{\delta^\sharp}^R,
\end{eqnarray}
since we have
\begin{eqnarray*}
w^\sharp (x,t)=u^*(x,t+\delta^\sharp)\;\;\mbox{for all}\;\;(x,t)\in (\Omega\times(0,T^*-\delta^\sharp])\setminus \mathcal{Q}_{\delta^\sharp}^R;
\end{eqnarray*}
\begin{eqnarray*}
w^\sharp (x,t)=\frac{\nu_1+\nu_2}{2} \;\;\mbox{for all}\;\;(x,t)\in\Omega\times(T^*-\delta^\sharp,+\infty).
\end{eqnarray*}

   To prove \eqref{wang-5-19-2-9}, we first use  \eqref{yu-9-11-b-1} to obtain
   \begin{eqnarray*}
\chi_{\mathcal{Q}_{\delta^\sharp}^R}(x,t)\leq \chi_{\mathcal{D}_{\delta^\sharp}}(x,t)=\chi_{\mathcal{D}}(x,t+\delta^{\sharp})
 \;\;\mbox{for all}\;\;(x,t)\in \Omega\times(0,T^*-\delta^\sharp].
\end{eqnarray*}
   Then by the above observation and  \eqref{yu-9-10-2}, we have
    \begin{eqnarray}\label{wang5-20-2-9}
    \nu_1+\varepsilon<u^*(x,t+\delta^\sharp)
    <\nu_2-\varepsilon\;\;\mbox{for all}\;\;(x,t)\in \mathcal{Q}_{\delta^\sharp}^R.
\end{eqnarray}
   Meanwhile,  by \eqref{yu-9-11-9} and \eqref{yu-9-11-11}, we see
   \begin{eqnarray}\label{wang5-20-2-10}
   -\frac{\varepsilon}{2}\leq f_{\delta^{\sharp}}(x,t)
    \leq \frac{\varepsilon}{2}\;\;\mbox{for a.e.}\;\;(x,t)\in \mathcal{Q}_{\delta^\sharp}^R.
\end{eqnarray}
   By \eqref{wang5-20-2-9}, \eqref{wang5-20-2-10} and \eqref{yu-9-11-12}, we obtain \eqref{wang-5-19-2-9},
   which implies \eqref{yu-9-11-13}.

   Finally,    it follows from the measurability of $w^\sharp$ and \eqref{yu-9-11-13} that $w^\sharp\in \mathcal{U}_{ad}$.
   Moreover, by  \eqref{yu-9-11-2}, \eqref{yu-9-11-3}, \eqref{yu-9-11-10} and \eqref{yu-9-11-12}, we obtain
    \eqref{yu-9-11-1}.
   This completes the proof.
       \end{proof}

\end{document}